\documentclass{amsart}
\usepackage{amssymb}
\usepackage{graphicx}
\usepackage{longtable} 

\theoremstyle{plain}
\newtheorem{theorem}{Theorem}[section]
\newtheorem{lemma}[theorem]{Lemma}
\newtheorem{proposition}[theorem]{Proposition}

\theoremstyle{definition}

\newtheorem{remark}[theorem]{Remark}
\theoremstyle{remark}

\newcommand{\bZ}{\mathbb{Z}}
\newcommand{\bQ}{\mathbb{Q}}
\newcommand{\bR}{\mathbb{R}}
\newcommand{\bC}{\mathbb{C}} 

\newcommand{\cO}{\mathcal{O}} 
 
\newcommand{\cA}{\mathcal{A}} 

\newcommand{\fO}{\mathfrak{O}} 
\newcommand{\fA}{\mathfrak{A}} 

\newcommand{\fH}{\mathfrak{H}}

\title[Dimension Formula for Siegel cusp forms]
{An explicit dimension formula for Siegel cusp forms with respect to 
the non-split symplectic groups} 
\author{Hidetaka Kitayama}

\begin{document}
\begin{abstract}
The purpose of this paper is to give an explicit dimension formula 
for the spaces of vector valued Siegel cusp forms of degree two 
with respect to a certain kind of arithmetic subgroups of the non-split $\bQ$-forms 
of $Sp(2,\bR)$. 
We obtain our result by using Hashimoto and Ibukiyama's results in \cite{HI80},\cite{HI83} 
and Wakatsuki's formula in \cite{Wak}. 
Our result is a generalization of formulae in \cite[Theorem 4.1]{Has84} and 
\cite[Theorem 6.1]{Wak}. 
\end{abstract}

\maketitle
\vspace*{-10mm}

%%%%%%%%%%%%%%%%%%%%%%%%%%%%%%%%%%%%%%%%%%%%%%%%%%%%%%%%%%%%%%%%
%%%%%%%%%%%%%%%%%%%%%%%%%%%%%%%%%%%%%%%%%%%%%%%%%%%%%%%%%%%%%%%%
%
\section{Introduction}\label{seintro}
%
%%%%%%%%%%%%%%%%%%%%%%%%%%%%%%%%%%%%%%%%%%%%%%%%%%%%%%%%%%%%%%%%
%%%%%%%%%%%%%%%%%%%%%%%%%%%%%%%%%%%%%%%%%%%%%%%%%%%%%%%%%%%%%%%%

Let $B$ be an indefinite division quaternion algebra over $\bQ$ 
with discriminant $D$. 
Let $\mathfrak{O}$ be the maximal order of $B$, 
which is unique up to conjugation. 
If we take a positive divisor $D_1$ of $D$ and put $D_2:=D/D_1$, 
then there is the unique maximal two-sided ideal $\mathfrak{A}$ of $\mathfrak{O}$ 
such that $\mathfrak{A}\otimes _{\bZ}\bZ _p= \mathfrak{O}_p$ if $p\mid D_1$ or $p\nmid D$,  
and $\mathfrak{A}\otimes _{\bZ}\bZ _p= \pi\mathfrak{O}_p$ if $p\mid D_2$, 
where $\pi$ is a prime element of $\mathfrak{O}_p$. 
We consider the unitary group of the quaternion hermitian space of rank two 
and denote by $\Gamma(D_1,D_2)$ the stabilizer of the maximal lattice 
$(\fA,\fO)$, namely we define 
\begin{center} 
\scalebox{0.9}[1]{  
$\Gamma (D_1,D_2):=\left\{ g=\begin{pmatrix} a&b\\c&d \end{pmatrix} \in GL(2;B) 
\ \left| \ g \begin{pmatrix} 0&1\\1&0 \end{pmatrix} {}^t \overline{g} = 
\begin{pmatrix} 0&1\\1&0 \end{pmatrix}, 
\begin{array}{l} a,d\in \fO, \\ b\in \fA ^{-1}, c\in \fA \end{array} 
\right. \right\}$. 
}   
\end{center} 
(See section \ref{subsec:def}.)  
This group can be regarded as a discrete subgroup of $Sp(2;\bR)$. 
\nolinebreak Our main theorem (Theorem \ref{thm:main}) is an explicit formula 
for dimension of $S_{k,j}(\Gamma (D_1,D_2))$, 
the space of vector valued Siegel cusp forms 
of weight $\mathrm{det}^k\otimes \mathrm{Sym}_j$ 
with respect to $\Gamma (D_1,D_2)$ for $k\geq 5$. 
For example, for a prime number $p\neq 2$,  we have 
\begin{align*} 
\mathrm{dim}_{\bC}S_{k,0}(\Gamma (1,2p)) &= 
\frac{(k-2)(k-1)(2k-3)}{2^7\cdot 3^2\cdot 5}\cdot (p^2-1) +\frac{1}{2^3\cdot 3}\cdot (p-1)\\ 
 &{\quad } +\frac{(-1)^k(8+\big( \frac{-1}{p}\big) )+(2k-3)(8-\big( \tfrac{-1}{p}\big) )}{2^7\cdot 3}( p-\big( \tfrac{-1}{p}\big) ) \\ 
 &{\quad } +\frac{[0,-1,1;3]_k}{2^2\cdot 3^2} \cdot \Big( 4+\tfrac{1}{2}\big( \tfrac{-3}{p}\big) \big( 1-5\big( \tfrac{-3}{p}\big) \big)\Big) \big( p-\big( \tfrac{-3}{p}\big) \big) \\  
 &{\quad } +\frac{2k-3}{2^2\cdot 3^2} \cdot \Big( 5-\tfrac{1}{2}\big( \tfrac{-3}{p}\big) \big( 1+7\big( \tfrac{-3}{p}\big) \big)\Big) \big( p-\big( \tfrac{-3}{p}\big) \big) \\ 
 &{\quad } -\frac{1}{2^3}( 1-\big( \tfrac{-1}{p}\big) ) 
-\frac{1}{3}( 1-\big( \tfrac{-3}{p}\big) ) \\ 
 &{\quad } +\frac{2\cdot [1,0,0,-1,0;5]_k}{5}\cdot ( 1-\big( \tfrac{p}{5}\big) ) \\ 
 &{\quad } +\frac{[1,0,0,-1;4]_k}{2^2}\cdot \left\{ \begin{array}{ccc} 0&\cdots &\mbox{ if }p\equiv 1,7\mbox{ mod }8 \\ 1&\cdots &\mbox{ if }p\equiv 3,5\mbox{ mod }8 \end{array} \right. \\   
 &{\quad } +\frac{1}{6}\cdot 
\left\{ \begin{array}{ccl} 
        (-1)^k/2&\cdots &\mbox{ if }p=3 \\ 
        0&\cdots &\mbox{ if }p\equiv 1,11\mbox{ mod }12 \\ 
        \left[ 0,1,-1;3\right] _k &\cdots &\mbox{ if }p\equiv 5\mbox{ mod }12 \\ 
        (-1)^k&\cdots &\mbox{ if }p\equiv 7\mbox{ mod }12,  
\end{array} \right. 
\end{align*} 
where $\left( \frac{*}{*}\right)$ is the Legendre symbol and 
$[a_0,\ldots ,a_{m-1};m]_k$ is the function on $k$ 
which takes the value $a_i$ if $k\equiv i\mbox{ mod }m$. 
This example is a formula for dimension of the space of scalar valued Siegel cusp forms 
of weight $k$ with respect to $\Gamma(1,2p)$.  

Explicit dimension formula 
for the spaces of Siegel cusp forms of degree two 
has been studied by many mathematicians. 
Among them, Arakawa \cite{Ara75},\cite{Ara81}, Hashimoto \cite{Has84} and Wakatsuki \cite[Theorem 6.1]{Wak} treated the non-split $\bQ$-forms.  
Hashimoto \cite{Has84} obtained an explicit dimension formula for 
scaler valued Siegel cusp forms for $\Gamma(D,1)$ and  
Wakatsuki \cite[Theorem 6.1]{Wak} generalized it to the vector valued Siegel cusp forms for $\Gamma (D,1)$. 
Our main result of this paper (Theorem \ref{thm:main}) is an explicit dimension formula for vecter valued Siegel cusp forms for $\Gamma (D_1,D_2)$. 
It is a generalization of \cite{Has84} and \cite[Theorem 6.1]{Wak}. 

Our motivations for this study are as follows. 
First, we are interested to study a possible correspondence between 
Siegel modular forms for different discrete subgroups 
by means of comparison of dimension formulae. 
In the case where $B$ is definite, 
Ibukiyama has been studying a generalization of 
Eichler-Jacquet-Langlands correspondence to the case of $Sp(2)$. 
(cf. \cite{Ibu84},\cite{Ibu85},\cite{HI85}).  
He conjectured the correspondence of discrete subgroups 
and obtained some relations of dimension formulae. 
Our dimension formula would be used for a similar comparison. 
Second, we are interested in an explicit construction of Siegel modular forms 
for $\Gamma (D_1,D_2)$.  
We expected that dimensions in the case of $D_2\neq 1$ are smaller than 
those in the case of $D_2=1$ for a fixed $D$. (cf. $H_1$ in Theorem \ref{thm:main}). 
We see that this is true from our main theorem. 
(cf. numerical examples in section \ref{sec:example}). 
In addition, we have succeeded in constructing  the graded ring of 
scaler valued Siegel modular forms for $\Gamma (1,6)$ explicitly 
by using our dimension formula, 
Rankin-Cohen type differential operator in \cite{Ibu99},\cite{AokIbu05}, and 
an explicit formula of Fourier coefficients of Eisenstein series in \cite{Hir99}.  
We will discuss it in another paper. 

We summarize the way to obtain our main result. 
We divide $\Gamma=\Gamma(D_1,D_2)$ into disjoint union of four subsets 
$\Gamma ^{(e)}$, $\Gamma ^{(u)}$, $\Gamma ^{(qu)}$ and $\Gamma ^{(h)}$ as follows: 

\vspace{2mm} 
(i)\ $\Gamma ^{(e)}$ consists of torsion elements of $\Gamma$. 

(ii)\ $\Gamma ^{(u)}$ consists of non-semi-simple elements of $\Gamma$ 
whose semi-simple factors are \\ \hspace{25pt} $1_4$ or $-1_4$. 

(iii)\ $\Gamma ^{(qu)}$ consists of non-semi-simple elements of $\Gamma$ 
whose semi-simple factors \\ \hspace{25pt} belong to $\Gamma^{(e)}$ other than $\pm 1_4$.  

(iv)\ $\Gamma ^{(h)}$ consists of the other elements of $\Gamma$ 
than the above three types. 
 
\vspace{2mm} \noindent 
We denote the contributions to dimension formula of each subset above by 
$I(\Gamma^{(e)})_{k,j}$, $I(\Gamma^{(u)})_{k,j}$, $I(\Gamma^{(qu)})_{k,j}$ 
and $I(\Gamma^{(h)})_{k,j}$. 
It is known that $I(\Gamma^{(h)})_{k,j}=0$ and 
\begin{align*} 
\dim_{\bC}S_{k,j}(\Gamma )=
I(\Gamma^{(e)})_{k,j}+I(\Gamma^{(u)})_{k,j}+I(\Gamma^{(qu)})_{k,j}. 
\end{align*} 
We can evaluate the contribution $I(\Gamma^{(e)})_{k,j}$ by using the method developed by 
Hashimoto and Ibukiyama, and the contribution $I(\Gamma^{(u)})_{k,j}$ and $I(\Gamma^{(qu)})_{k,j}$ by using formulae of Wakatsuki \cite{Wak}.  
It is known by Hashimoto and Ibukiyama \cite{HI80}, \cite{Has83} 
that the formula for $I(\Gamma^{(e)})_{k,j}$ 
can be expressed adelically and can be reduced to local computation 
(cf. Theorem \ref{thm:elliptic}). 
The crucial work in using Theorem \ref{thm:elliptic} is to evaluate the local data 
^^ ^^ $c_p(g,R_p,\Lambda_p)$", 
and this work has been completed by Hashimoto and Ibukiyama 
in \cite{HI80} and \cite{HI83}. 
So, we have only to combine the data depending on the cases, 
but still it is a complicated work. 
Hashimoto and Ibukiyama gave the result in the case where $B$ is definite 
in \cite{HI80} and \cite{HI82}. 
Our case is where $B$ is indefinite and the way of combining local data 
is different from that of \cite{HI80} and \cite{HI82}. 
We will explain the detail in section \ref{sec:elliptic}. 
On the other hand, as for the non-semi-simple conjugacy classes, 
we can not reduce the contributions of them to local calculations. 
Wakatsuki \cite{Wak} gave an arithmetic formula for the contributions of them, 
but one still have to carry out detailed calculation to obtain an explicit formula.  
More precisely, we need to determine a complete system of representatives of 
$\Gamma$-conjugacy classes of ^^ ^^ families" 
(cf. Proposition \ref{Proposition:Unipotent1}, \ref{Prop:ClassifyQU}, \ref{prop:classifyI3}) 
and calculate some data for them.  
Arakawa has calculated the contribution $I(\Gamma^{(u)})$ 
in his master thesis \cite{Ara75}. 
We prove it again in Section \ref{sec:nonsemisimple} by means of Wakatsuki's formula (e-2). 
Hashimoto calculated the contribution $I(\Gamma^{(qu)})$ 
in the case where $D_1=D$ and $D_2=1$ in \cite{Has84}. 
We can calculate it also in the general case 
by almost the same way. 

We organize this paper as follows. 
In section \ref{sec:Pre}, 
we review Siegel cusp forms in subsection \ref{subsec:Siegel} and 
explain the definition of the discrete subgroup $\Gamma (D_1,D_2)$ which we treat 
in this paper in subsection \ref{subsec:nonsplit} and \ref{subsec:def}. 
In section \ref{sec:main}, 
we state our main theorem (Theorem \ref{thm:main}).  
We prove Theorem \ref{thm:main} in section \ref{sec:elliptic} and \ref{sec:nonsemisimple}. 
In section \ref{sec:elliptic}, we evaluate the contribution $I(\Gamma^{(e)})_{k,j}$. 
First, we quote the formula of Hashimoto and Ibukiyama 
(Theorem \ref{thm:elliptic}), 
and then we prove $H_1$, $\ldots$, $H_{12}$ of Theorem \ref{thm:main} 
in subsection \ref{subsec:H1} -- \ref{subsec:H12}.  
In section \ref{sec:nonsemisimple}, 
we evaluate the contribution $I(\Gamma^{(u)})_{k,j}$ and $I(\Gamma^{(qu)})_{k,j}$. 
We prove $I_1$, $I_2$ and $I_3$ of Theorem \ref{thm:main} in subsection \ref{subsec:I1}, \ref{subsec:I2} and \ref{subsec:I3} respectively. 
In section \ref{sec:example}, 
we give some numerical examples for our main theorem. 
 
The author would like to express his sincere gratitude to Professor Tomoyoshi Ibukiyama for giving him this problem and various advices. 
The author also would like to thank Professor Takashi Sugano and 
Professor Satoshi Wakatsuki for their many helpful advices. 
The author is supported by the Grant-in-Aid for JSPS fellows.

%%%%%%%%%%%%%%%%%%%%%%%%%%%%%%%%%%%%%%%%%%%%%%%%%%%%%%%%%%%%%%%%
%%%%%%%%%%%%%%%%%%%%%%%%%%%%%%%%%%%%%%%%%%%%%%%%%%%%%%%%%%%%%%%%
%
\section{Preliminaries} 
\label{sec:Pre} 
%
%%%%%%%%%%%%%%%%%%%%%%%%%%%%%%%%%%%%%%%%%%%%%%%%%%%%%%%%%%%%%%%%
%%%%%%%%%%%%%%%%%%%%%%%%%%%%%%%%%%%%%%%%%%%%%%%%%%%%%%%%%%%%%%%% 

%%%%%%%%%%%%%%%%%%%%%%%%%%%%%%%%%%%%%%%%%%%%%%%%%%%%%%%%%%%%%%%%
%
\subsection{Siegel cusp forms} 
\label{subsec:Siegel} 
%
%%%%%%%%%%%%%%%%%%%%%%%%%%%%%%%%%%%%%%%%%%%%%%%%%%%%%%%%%%%%%%%%

Let $Sp(2;\bR)$ be the real symplectic group of degree two, i.e. 
\[ Sp(2;\bR ) =\left\{ g\in GL(4,\bR )\ \left| \ g\begin{pmatrix} 0_2&1_2\\-1_2&0_2 \end{pmatrix} {}^tg=\begin{pmatrix} 0_2&1_2\\-1_2&0_2 \end{pmatrix} \right. \right\} . \] 
Let $\fH_2$ be the Siegel upper half space of degree two, i.e. 
\[ \fH _2=\{ Z\in M(2;\bC )\ |\ {}^tZ=Z,\ \mathrm{Im}(Z)\mbox{ is positive definite } \} . \] 
The group $Sp(2;\bR)$ acts on $\fH_2$ by 
\[ \gamma \langle Z\rangle :=(AZ+B)(CZ+D)^{-1} \] 
for any $\gamma =\begin{pmatrix} A&B\\C&D \end{pmatrix}\in Sp(2;\bR)$ 
and $Z\in \fH_2$.    

Let $\Gamma$ be a discrete subgroup of $Sp(2;\bR)$ such that vol($\Gamma\backslash \fH_2)<\infty$. 
Let $\rho_{k,j}:GL(2;\bC)\rightarrow GL(j+1;\bC)$ be the irreducible rational 
representation of the signature $(j+k,k)$ for $k,j\in\bZ_{\geq 0}$, 
i.e. $\rho _{k,j}=\det ^k\otimes Sym _j$, 
where $Sym _j$ is the symmetric $j$-tensor representation of $GL(2;\bC)$. 
We denote by $S_{k,j}(\Gamma)$ the space of Siegel cusp forms of weight $\rho _{k,j}$
 with respect to $\Gamma$, 
i.e. the space which consists of holomorphic function $f: H_2\rightarrow \bC ^{j+1}$ 
satisfying the following two conditions: \\ 
\hspace{10mm}(i)\ $f(\gamma \langle Z\rangle )=\rho _{k,j}(CZ+D)f(Z)$, \hspace{5mm} 
for $\forall \gamma =\begin{pmatrix} A&B\\C&D \end{pmatrix} 
\in \Gamma$, $\forall Z\in \fH_2$, \\ 
\hspace{9mm} (ii)\ $\big| \rho _{k,j}(\mathrm{Im}(Z)^{1/2})f(Z) \big| _{\bC ^{j+1}}$ is bounded on $\fH_2$, \\ 
where we define $|u|_{\bC ^{j+1}}=({}^tu\overline{u})^{\frac{1}{2}}$
for $u\in \bC^{j+1}$. 
It is known that $S_{k,j}(\Gamma)$ is a finite dimensional $\bC$-vector space.

%%%%%%%%%%%%%%%%%%%%%%%%%%%%%%%%%%%%%%%%%%%%%%%%%%%%%%%%%%%%%%%%
%
\subsection{The non-split $\bQ$-forms of $Sp(2;\bR)$} 
\label{subsec:nonsplit} 
%
%%%%%%%%%%%%%%%%%%%%%%%%%%%%%%%%%%%%%%%%%%%%%%%%%%%%%%%%%%%%%%%%  

Let $B$ be an indefinite quaternion algebra over $\bQ$. 
We fix an isomorphism $B\otimes _{\bQ}\bR\simeq M(2;\bR)$  
and we identify $B$ with a subalgebra of $M(2;\bR)$. 
Let $D$ be a product of all prime numbers $p$ for which $B\otimes _{\bQ}\bQ_p$ is 
a division algebra. 
We call $D$ the discriminant of $B$. 
Let $W$ be a left free $B$-module of rank 2. 
Let $f$ be a map on $W\times W$ to $B$ defined by 
\begin{align*} 
f(x,y)=x_1\overline{y_2}+x_2\overline{y_1},\quad 
x=(x_1,x_2),y=(y_1,y_2)\in W, 
\end{align*} 
where $\bar{\ }$ is the canonical involution of $B$. 
Any non-degenerate quaternion hermitian form on $W$ is equivalent to $f$. 
(cf. \cite{Shi63}). 
Let $U(2;B)$ be the unitary group with respect to this hermitian space $(W,f)$, 
that is, 
\begin{align*} 
U(2;B) &= \{ g\in GL(2;B)\ |\ f(xg,yg)=f(x,y)\ \mbox{for}\ \forall x,y\in W\} \\ 
       &= \left\{ g\in GL(2;B)\ \left| \ g\begin{pmatrix} 0&1\\1&0 \end{pmatrix} 
{}^t\overline{g} = \begin{pmatrix} 0&1\\1&0 \end{pmatrix} \right. \right\} , 
\end{align*}
where ${}^t\overline{g}=\begin{pmatrix} \overline{a}&\overline{c} \\ \overline{b}&\overline{d} \end{pmatrix} 
\mbox{ for } g=\begin{pmatrix} a&b \\ c&d \end{pmatrix}$. 
It is known that $U(2;B)\otimes _{\bQ}\bR$ is isomorphic to $Sp(2;\bR)$ by  
\[ \phi : U(2;B)\otimes _{\bQ}\bR \stackrel{\sim}{\longrightarrow}Sp(2;\bR) \] 
\[ \phi (g)=\begin{pmatrix} a_1&a_2&b_2&-b_1\\a_3&a_4&b_4&-b_3\\c_3&c_4&d_4&-d_3\\-c_1&-c_2&-d_2&d_1 \end{pmatrix} ,\quad   
g=\begin{pmatrix} A&B\\C&D \end{pmatrix} \in U(2;B)\otimes _{\bQ}\bR \] 
where $A=\begin{pmatrix} a_1&a_2\\a_3&a_4 \end{pmatrix}$, 
$B=\begin{pmatrix} b_1&b_2\\b_3&b_4 \end{pmatrix}$, 
$C=\begin{pmatrix} c_1&c_2\\c_3&c_4 \end{pmatrix}$, 
$D=\begin{pmatrix} d_1&d_2\\d_3&d_4 \end{pmatrix} \in B\otimes _{\bQ}\bR$. 
It is known that each $\bQ$-form of $Sp(2;\bR)$ can be obtained as $U(2;B)$ 
for some indefinite quaternion algebra $B$ (cf. \cite{PR94}). 
If $B=M(2;\bQ)$, then $U(2;B)$ is isomorphic to $Sp(2;\bQ)$ by $\phi$. 
In this paper, we treat the case where $B$ is a division algebra. 

%%%%%%%%%%%%%%%%%%%%%%%%%%%%%%%%%%%%%%%%%%%%%%%%%%%%%%%%%%%%%%%%
%
\subsection{} 
\label{subsec:def} 
%
%%%%%%%%%%%%%%%%%%%%%%%%%%%%%%%%%%%%%%%%%%%%%%%%%%%%%%%%%%%%%%%%

Let $\fO$ be the maximal order of $B$, which is unique up to inner automorphisms. 
We fix a quaternion hermitian space $(W,f)$. 
Let $L$ be a left $\fO$-lattice in $W$, that is, 
$L$ is a finitely generated $\bZ$-module satisfying $L\otimes _{\bZ}\bQ =W$ and 
$aL\subset L$ for any $a\in \fO$. 
We put 
\begin{align*} 
U(2;B)_L:=\{ g\in U(2;B)\ |\ Lg=L\}. 
\end{align*} 
Then it is a discrete subgroup of $Sp(2,\bR)$ such that 
$\mathrm{vol}(U(2;B)_L\backslash \fH_2)<\infty$ by identifying it with its image 
by $\phi$ in $Sp(2;\bR)$. 

The two-sided $\fO$-ideal generated by the elements $f(x,y)$ for $x,y\in L$ 
is called the norm of $L$. 
We call $L$ a maximal lattice if $L$ is maximal among the left $\fO$-lattices 
having the same norm. 
For any maximal lattice $L$ and any prime number $p$, 
it is known by \cite{Shi63} that 
\begin{align*} 
L\otimes _{\bZ}\bZ _p=
\begin{cases} 
\begin{array}{ll} 
(\fO_p,\fO_p)g_p & \cdots \mbox{ if } p\nmid D \\ 
(\fO_p,\fO_p)g_p \mbox{ or }(\pi\fO_p,\fO_p)g_p & \cdots \mbox{ if } p\mid D 
\end{array} 
\end{cases} 
\end{align*} 
for some $g_p\in U(2;B)\otimes _{\bQ} \bQ_p$, 
where $\fO_p:=\fO\otimes \bZ_p$ and $\pi$ is a prime element of $\fO_p$. 
Hence there are exactly $2^s$ genera of maximal lattices in $W$ 
if $D$ is a product of $s$ prime numbers. 
We put $D=D_1D_2$, where $D_1,D_2\in\mathbb{N}$ such that 
$L\otimes _{\bZ}\bZ_p=(\fO_p,\fO_p)g_p$ if $p\mid D_1$, 
and $L\otimes _{\bZ}\bZ_p=(\pi\fO_p,\fO_p)g_p$ if $p\mid D_2$, 
for some $g_p\in U(2;B)\otimes _{\bQ}\bQ_p$. 
It is known that if two maximal lattices $L_1$ and $L_2$ correspond to the same pair 
$(D_1,D_2)$, then $L_1$ and $L_2$ belong to the same class (i.e. $L_1=L_2g$ for 
some $g\in U(2;B)$ ) since $B$ is indefinite, 
and therefore $U(2;B)_{L_1}=U(2;B)_{L_2}$. 
For simplicity, we put 
\begin{align*} 
\Gamma (D_1,D_2) := U(2;B)_L 
\end{align*} 
for the maximal lattice $L$ corresponding to the pair $(D_1,D_2)$.

%%%%%%%%%%%%%%%%%%%%%%%%%%%%%%%%%%%%%%%%%%%%%%%%%%%%%%%%%%%%%%%%
%%%%%%%%%%%%%%%%%%%%%%%%%%%%%%%%%%%%%%%%%%%%%%%%%%%%%%%%%%%%%%%%
%
\section{Main result} 
\label{sec:main} 
%
%%%%%%%%%%%%%%%%%%%%%%%%%%%%%%%%%%%%%%%%%%%%%%%%%%%%%%%%%%%%%%%%
%%%%%%%%%%%%%%%%%%%%%%%%%%%%%%%%%%%%%%%%%%%%%%%%%%%%%%%%%%%%%%%%

Our main result is Theorem \ref{thm:main} below. 
It is an explicit dimension formula of the spaces of Siegel cusp forms 
of weight $\rho_{k,j}$ with respect to $\Gamma (D_1,D_2)$ defined above. 
This formula is a generalization of \cite{Has84} and \cite[Theorem 6.1]{Wak}. 
We prove Theorem \ref{thm:main} in sections \ref{sec:elliptic} and \ref{sec:nonsemisimple}. 
We suppose that $j$ is even. 
If $j$ is odd, we have $S_{k,j}(\Gamma (D_1,D_2))=\{ 0\}$ for any $k$ 
since $\Gamma (D_1,D_2)$ contains $-1_4$. 
For natural number $m$ and $n$, 
we denote by $[a_0,\ldots ,a_{m-1};m]_n$ the function on $n$ 
which takes the value $a_i$ if $n\equiv i\mbox{ mod }m$. 
We define the set $T(m;n):=\{ p\mid T\ ;\ p\equiv m\mbox{ mod }n\}$ 
for $T=D$, $D_1$ or $D_2$. 

\begin{theorem} \label{thm:main} 
If $k\ge 5$ and $j$ is an even non-negative integer, then we have 
\begin{align*} 
\dim _{\bC} S_{k,j}(\Gamma (D_1,D_2))= 
\sum _{i=1}^{12} H_{i} + \sum _{i=1}^3 I_i,  
\end{align*} 
where $H_i$ and $I_i$ are given as follows:  
\end{theorem} 

\begin{longtable}{ccl}  
$H_1$ & $=$ & $\displaystyle \frac{(j+1)(k-2)(j+k-1)(j+2k-3)}{2^7\cdot 3^3\cdot 5} 
\cdot \prod _{p\mid D_1}(p-1)(p^2+1)\cdot \prod _{p\mid D_2}(p^2-1)$ \\ 
 & & \\  
$H_2$ & $=$ & $\displaystyle \frac{(-1)^k(j+k-1)(k-2)}{2^7\cdot 3^2} \cdot 
\prod _{p\mid D} (p-1)^2 \times 
\left\{ \begin{array}{ccl} 7&\cdots &\mbox{if } 2\nmid  D_1, D_2=1 \\ 
13&\cdots &\mbox{if } 2\mid  D_1, D_2=1 \\ 
3&\cdots &\mbox{if } D_2=2 \\ 
0&\cdots &\mbox{otherwise} \end{array} \right.$  \\ 
 & & \\ 
$H_3$ & $=$ & $\displaystyle \frac{[(-1)^{\frac{j}{2}} (k-2), -(j+k-1), (-1)^{\frac{j}{2}+1} (k-2), j+k-1; 4]_k}{2^5\cdot 3}$ \\ 
 & & $\displaystyle \times \prod _{p\mid D_1} (p-1)\Big( 1-\left( \frac{-1}{p}\right) \Big) 
\times \left\{ \begin{array}{ccl} 
1&\cdots &\mbox{if } D_2=1 \\ 
3&\cdots &\mbox{if } D_2=2 \\ 
0&\cdots &\mbox{otherwise} \end{array} \right.$  \\ 
 & &  \\ 
$H_4$ & $=$ & $\displaystyle \frac{[j+k-1, -(j+k-1), 0; 3]_k +[k-2, 0, -(k-2); 3]_{j+k}}{2^3\cdot 3^3}$ \\ 
 & & $\displaystyle \times \prod _{p\mid D_1} (p-1)\Big( 1-\left( \frac{-3}{p}\right) \Big) 
\times \left\{ \begin{array}{ccl} 
1&\cdots &\mbox{if } D_2=1 \\ 
8&\cdots &\mbox{if } D_2=3 \\ 
0&\cdots &\mbox{otherwise} \end{array} \right.$ \\ 
 & & \\ 
$H_5$ & $=$ &$\displaystyle 2^{-3}\cdot 3^{-2}\cdot \big([-(j+k-1),-(j+k-1),0,j+k-1,j+k-1,0;6]_k$ \\ 
 & & \hspace{5cm} $+[k-2,0,-(k-2),-(k-2),0,k-2;6]_{j+k} \big)$   \\ 
 & & $\displaystyle \times \prod _{p\mid D_1} (p-1)\Big( 1-\left( \frac{-3}{p}\right) \Big) 
\times \left\{ \begin{array}{ccl} 
1&\cdots &\mbox{if } D_2=1 \\ 
0&\cdots &\mbox{otherwise} \end{array} \right.$ \\ 
 & & \\ 
$H_6$ & $=$ & $\displaystyle \sum\limits_{n\mid 2D_1}A \ 
\prod\limits_{p\mid n}(p-1) \prod\limits_{{\tiny \begin{array}{c} p\nmid n \\ p\mid D_2 \\ p\neq 2 \end{array}}} (p+1) \prod\limits_{{\tiny \begin{array}{c} p\nmid n\\ p\mid D_1 \\ p\neq 2 \end{array}}} 2 \cdot B$ \\ 
 & & For each $n$, $A$ and $B$ are defined as follows ;  \\ 
 & & $A=\left\{ \begin{array}{lcl} 
     2^{-7}3^{-1}(-1)^{k+j/2}(j+1) & \cdots & \mbox{if $n$ has odd numbers of prime divisors} \\ 
     2^{-7}3^{-1}(-1)^{j/2}(j+2k-3) & \cdots & \mbox{if $n$ has even numbers of prime divisors} \end{array} \right.$. \\ 
 & & If $D$ has a prime divisor $p$ such that $\left( \frac{-1}{p}\right) =1$, then $B=0$,  \\  
 & & otherwise,  \\ 
 & & \qquad \qquad $B = \left\{ \begin{array}{ccl} 
5&\cdots &\mbox{if }2\mid D_1\mbox{ and }2\mid n \\ 
11&\cdots &\mbox{if }2\mid D_1\mbox{ and }2\nmid n \\ 
7&\cdots &\mbox{if }2\mid D_2\mbox{ and }2\mid n \\ 
9&\cdots &\mbox{if }2\mid D_2\mbox{ and }2\nmid n \\ 
3&\cdots &\mbox{if }2\nmid D\mbox{ and }2\mid n \\ 
5&\cdots &\mbox{if }2\nmid D\mbox{ and }2\nmid n \end{array} \right. $  \\ 
 & & \\ 
$H_7$ & $=$ & $\displaystyle \sum\limits_{n\mid 3D_1}A \ 
\prod\limits_{p\mid n}(p-1) \prod\limits_{{\tiny \begin{array}{c} p\nmid n \\ p\mid D_2 \\ p\neq 3 \end{array}}} (p+1) \prod\limits_{{\tiny \begin{array}{c} p\nmid n\\ p\mid D_1 \\ p\neq 3 \end{array}}} 2 \cdot B$ \\ 
 & & For each $n$, $A$ and $B$ are defined as follows ;  \\ 
 & & $A=\left\{ \begin{array}{lcl} 
     2^{-3}3^{-3}(j+1)[0,1,-1;3]_{j+2k} & \cdots & \mbox{if $n$ has odd numbers of prime divisors} \\ 
     2^{-3}3^{-3}(j+2k-3)[1,-1,0;3]_j & \cdots & \mbox{if $n$ has even numbers of prime divisors} \end{array} \right.$. \\ 
 & & If $D$ has a prime divisor $p$ such that $\left( \frac{-3}{p}\right) =1$, then $B=0$,  \\  
 & & otherwise,  \\ 
 & & \qquad \qquad $B = \left\{ \begin{array}{ccl} 
1&\cdots &\mbox{if }3\mid D_1\mbox{ and }3\mid n \\ 
16&\cdots &\mbox{if }3\mid D_1\mbox{ and }3\nmid n \\ 
4&\cdots &\mbox{if }3\mid D_2\mbox{ and }3\mid n \\ 
10&\cdots &\mbox{if }3\mid D_2\mbox{ and }3\nmid n \\ 
1&\cdots &\mbox{if }3\nmid D\mbox{ and }3\mid n \\ 
4&\cdots &\mbox{if }3\nmid D\mbox{ and }3\nmid n \end{array} \right. $  \\ 
 & & \\ 
$H_8$ & $=$ & $\displaystyle \frac{C_1}{2^2\cdot 3}\cdot \prod _{p\mid D} \Big( 1-\left( \frac{-1}{p}\right) \Big)  \Big( 1-\left( \frac{-3}{p}\right) \Big) 
\times \left\{ \begin{array}{ccl} 
1&\cdots &\mbox{if }D_2=1 \\ 
0&\cdots &\mbox{otherwise} \end{array} \right. $, \\ 
 & & where we put \\ 
 & & $C_1=\left\{ \begin{array}{ccl} 
{[1,0,0,-1,-1,-1,-1,0,0,1,1,1;12]_k}&\cdots &\mbox{ if }j\equiv 0 \mbox{ mod }12 \\ 
{[-1,1,0,1,1,0,1,-1,0,-1,-1,0;12]_k}&\cdots &\mbox{ if }j\equiv 2 \mbox{ mod }12 \\ 
{[1,-1,0,0,-1,1,-1,1,0,0,1,-1;12]_k}&\cdots &\mbox{ if }j\equiv 4 \mbox{ mod }12 \\ 
{[-1,0,0,-1,1,-1,1,0,0,1,-1,1;12]_k}&\cdots &\mbox{ if }j\equiv 6 \mbox{ mod }12 \\ 
{[1,1,0,1,-1,0,-1,-1,0,-1,1,0;12]_k}&\cdots &\mbox{ if }j\equiv 8 \mbox{ mod }12 \\ 
{[-1,-1,0,0,1,1,1,1,0,0,-1,-1;12]_k}&\cdots &\mbox{ if }j\equiv 10 \mbox{ mod }12 
\end{array} \right.$.  \\ 
 & & \\ 
$H_9$ & $=$ & $\displaystyle \frac{C_2}{2\cdot 3^2}\times \prod _{p\mid D_1,p\not=2} \Big( 1-\left( \frac{-3}{p}\right) \Big) ^2 \times 
\left\{ \begin{array}{ccl} 
2&\cdots &\mbox{if } 2\nmid  D_1 \mbox{ and } D_2=1 \\ 
5&\cdots &\mbox{if } 2\mid D_1 \mbox{ and } D_2=1 \\ 
3&\cdots &\mbox{if } 2\nmid  D_1 \mbox{ and } D_2=2 \\ 
0&\cdots &\mbox{otherwise } \end{array} \right. $, \\ 
 & & where we put \\ 
 & & $C_2=\left\{ \begin{array}{ccl} 
{[1,0,0,-1,0,0;6]_k}&\cdots &\mbox{if }j\equiv 0 \mbox{ mod } 6 \\  
{[-1,1,0,1,-1,0;6]_k}&\cdots &\mbox{if }j\equiv 2 \mbox{ mod } 6 \\  
{[0,-1,0,0,1,0;6]_k}&\cdots &\mbox{if }j\equiv 4 \mbox{ mod } 6   
\end{array} \right.$. \\ 
 & & \\ 
$H_{10}$ & $=$ & $\displaystyle \frac{C_3}{2\cdot 5}\times \prod _{p\mid D}2\times \prod _{p\in D(4;5)} 2 
\times$ 
\scalebox{0.8}[0.9]{ 
$\left\{ \begin{array}{ccl} 
0&\cdots &\mbox{if }\displaystyle \bigcup _{i=1}^3 D_1(i;5) \cup 
\bigcup _{i\in \{ 1,-1\} }D_2(i;5)  \not= \emptyset \\ 
1&\cdots &\mbox{if }\displaystyle  \bigcup _{i=1}^3 D_1(i;5) \cup 
\bigcup _{i\in \{ 1,-1\} }D_2(i;5)  = \emptyset \mbox{ and } 5\mid D \\ 
2&\cdots &\mbox{if }\displaystyle  \bigcup _{i=1}^3 D_1(i;5) \cup  
\bigcup _{i\in \{ 1,-1\} }D_2(i;5)  = \emptyset \mbox{ and } 5\nmid D 
\end{array} \right. $}, \\ 
 & & where we put \\ 
 & & $C_3=\left\{ \begin{array}{ccl} 
{[1,0,0,-1,0;5]_k}&\cdots & \mbox{if }j\equiv 0 \mbox{ mod } 10 \\ 
{[-1,1,0,0,0;5]_k}&\cdots & \mbox{if }j\equiv 2 \mbox{ mod } 10 \\ 
0&\cdots & \mbox{if }j\equiv 4 \mbox{ mod } 10 \\ 
{[0,0,0,1,-1;5]_k}&\cdots & \mbox{if }j\equiv 6 \mbox{ mod } 10 \\ 
{[0,-1,0,0,1;5]_k}&\cdots & \mbox{if }j\equiv 8 \mbox{ mod } 10 
\end{array} \right.$.   \\ 
 & & \\  
$H_{11}$ & $=$ & $\displaystyle \frac{C_4}{2^3}\times \prod _{p\mid D,p\not=2}2\times \prod _{p\in D_1(7;8)} 2 
\times \left\{ \begin{array}{ccl} 
0&\cdots &\mbox{if } D(1;8)\sqcup D_2(7;8) \not= \emptyset \\ 
1&\cdots &\mbox{otherwise}  
\end{array} \right. $, \\ 
 & & where we put \\ 
 & & $C_4=\left\{ \begin{array}{ccl} 
{[1,0,0,-1;4]_k}&\cdots & \mbox{if }j\equiv 0 \mbox{ mod } 8 \\ 
{[-1,1,0,0;4]_k}&\cdots & \mbox{if }j\equiv 2 \mbox{ mod } 8 \\ 
{[-1,0,0,1;4]_k}&\cdots & \mbox{if }j\equiv 4 \mbox{ mod } 8 \\ 
{[1,-1,0,0;4]_k}&\cdots & \mbox{if }j\equiv 6 \mbox{ mod } 8 
\end{array} \right.$.   \\    
 & &  \\ 
$H_{12}$ & $=$ & $\displaystyle 2^{-2}3^{-1}(-1)^{j/2+k}[1,-1,0;3]_j\times \prod _{p\mid D}2\times  \prod _{p\in D_1(11;12)}2\times A$ \\[20pt]  
 & & \hspace{5pt} $\displaystyle +2^{-2}3^{-1}(-1)^{j/2}[0,-1,1;3]_{j+2k}\times \prod _{p\mid D}2\times  \prod _{p\in D_1(11;12)}2\times B$, \\ 
 & & where $A$ and $B$ are defined as follows. \\ 
 & & (i) If $D(1;12)\sqcup D_2(11;12)\not= \emptyset$, then $A=B=0$. \\ 
 & & (ii) If $D(1;12)\sqcup D_2(11;12)= \emptyset$, then $A$ (resp.$B$) are given by the following table, \\  
 & & depending on  the conditions of $D$, $D_1$ and $D_2$. \vspace{4.62pt} \\ 
 & & \hspace{2cm} \begin{tabular}{|c|c|c|c|} \hline 
 & case (I) & case (II) & case (III)  \\ \hline 
$2\nmid D,3\nmid D$ & 0 & 1/2 & 1 \\ \hline 
$2\nmid D,3\mid D_1$ & 1/2 & 3/4 & 1 \\ \hline 
$2\nmid D,3\mid D_2$ & 0 & 1/4 & 1/2 \\ \hline 
$2\mid D_1,3\nmid D$ & 1 & 3/4 & 1/2 \\ \hline 
$2\mid D_1,3\mid D_1$ & 5/4 & 9/8 & 1 \\ \hline 
$2\mid D_1,3\mid D_2$ & 1/2 & 3/8 & 1/4 \\ \hline 
$2\mid D_2,3\nmid D$ & 1/2 & 1/4 & 0 \\ \hline 
$2\mid D_2,3\mid D_1$ & 1/2 & 3/8 & 1/4 \\ \hline 
$2\mid D_2,3\mid D_2$ & 1/4 & 1/8 & 0 \\ \hline 
\end{tabular} \vspace{4.62pt}  \\ 
 & & where case (I), (II) and (III) are given as follows: \\ 
 & & \hspace{2cm} $\left\{ \begin{array}{l} 
\mbox{(I) }D_1(11;12)=\emptyset \mbox{ and } \sharp D(5;12)=\mbox{even (resp. odd)} \\ 
\mbox{(II) }D_1(11;12)\not= \emptyset \\ 
\mbox{(III) }D_1(11;12)=\emptyset \mbox{ and } \sharp D(5;12)=\mbox{odd (resp. even)} 
\end{array} \right. $.   \\ 
 & & \\ 
$I_1$ & $=$ & $\displaystyle \frac{j+1}{2^3\cdot 3}\prod _{p\mid D} (p-1)$ \\ 
 & & \\ 
$I_2$ & $=$ & $\displaystyle -\frac{(-1)^{\frac{j}{2}}}{2^3} \prod _{p\mid D} \Big( 1-\left( \frac{-1}{p}\right) \Big)$ \\ 
 & & \\ 
$I_3$ & $=$ & $\displaystyle -\frac{[1,-1,0;3]_j}{2\cdot 3} \prod _{p\mid D} \Big( 1-\left( \frac{-3}{p}\right) \Big)$. 
\end{longtable}

%%%%%%%%%%%%%%%%%%%%%%%%%%%%%%%%%%%%%%%%%%%%%%%%%%%%%%%%%%%%%%%%%%%%%%%%%%
%%%%%%%%%%%%%%%%%%%%%%%%%%%%%%%%%%%%%%%%%%%%%%%%%%%%%%%%%%%%%%%%%%%%%%%%%%
%
\section{The contribution of semi-simple conjugacy classes} 
\label{sec:elliptic} 
%
%%%%%%%%%%%%%%%%%%%%%%%%%%%%%%%%%%%%%%%%%%%%%%%%%%%%%%%%%%%%%%%%%%%%%%%%%%
%%%%%%%%%%%%%%%%%%%%%%%%%%%%%%%%%%%%%%%%%%%%%%%%%%%%%%%%%%%%%%%%%%%%%%%%%% 

%We prove our main theorem (Theorem \ref{thm:main}) in the remainder of this paper. 
In this section, we evaluate $I(\Gamma^{(e)})_{k,j}$, i.e. 
the contributions of torsion elements (cf. section \ref{seintro}). 
The principal polynomials of torsion elements of $G$ $=$ $U(2;B)$ are as follows,  
and each $H_i$ in Theorem \ref{thm:main} means the contribution    
of $\Gamma$-conjugacy classes whose principal polynomials are of the form 
$f_i(\pm x)$.   
\begin{center} 
\scalebox{1}[0.9]{ 
\begin{tabular}{ll}  
$f_1(x) = (x-1)^4$, & $f_1(-x) = (x+1)^4$, \\ 
$f_2(x) = (x-1)^2(x+1)^2$, &  \\ 
$f_3(x) = (x-1)^2(x^2+1)$, & $f_3(-x) = (x+1)^2(x^2+1)$, \\ 
$f_4(x) = (x-1)^2(x^2+x+1)$, & $f_4(-x) = (x+1)^2(x^2-x+1)$, \\ 
$f_5(x) = (x-1)^2(x^2-x+1)$, & $f_5(-x) = (x+1)^2(x^2+x+1)$, \\ 
$f_6(x) = (x^2+1)^2$, &   \\ 
$f_7(x) = (x^2+x+1)^2$, & $f_7(-x) = (x^2-x+1)^2$, \\ 
$f_8(x) = (x^2+1)(x^2+x+1)$, & $f_8(-x) = (x^2+1)(x^2-x+1)$, \\ 
$f_9(x) = (x^2+x+1)(x^2-x+1)$, &    \\ 
$f_{10}(x) = x^4+x^3+x^2+x+1$, & $f_{10}(-x) = x^4-x^3+x^2-x+1$, \\ 
$f_{11}(x) = x^4+1$, &  \\ 
$f_{12}(x) = x^4-x^2+1$.  &  
\end{tabular}
}  
\end{center} 
\noindent 
We will evaluate each $H_i$ in subsection \ref{subsec:H1} -- \ref{subsec:H12}.  
The method was developed by Hashimoto and Ibukiyama 
\cite{Has80}, \cite{HI80}, \cite{HI82}, \cite{HI83}, \cite{Has83}, \cite{Has84}. 
We quote the formula for $I(\Gamma ^{(e)})_{k,j}$. 

\begin{theorem} \label{thm:elliptic} 
\begin{align*} 
I(\Gamma^{(e)})_{k,j}= c_{k,j} \sum _{\{ g\} _G} J'_0(g) 
\sum _{L_G(\Lambda )} 
M_G(\Lambda ) \prod _p c_p(g,R_p,\Lambda _p), 
\end{align*}  
where notations are as follows: 
\end{theorem} 
\hspace{-13.86pt} 
(i)\ The first sum is extended over the $G$-conjugacy classes $\{ g\} _G$  
of torsion elements of $G$ which satisfies $\{ g\} _G\cap\Gamma \neq \emptyset$. \\ 
(ii)\ The second sum is extended over the $G$-genus $L_G(\Lambda)$ of $\bZ$-orders 
in $Z(g)$ for each $g\in G$. 
Here $Z(g)$ is the commutor algebra of $g$ in $M(2;B)$. 
The $G$-genus $L_G(\Lambda)$ represented by a $\bZ$-order $\Lambda$ of $Z(g)$ 
consists of all $\bZ$-orders in $Z(g)$ 
which are $(Z(g)^{\times}\cap G)\otimes _{\bQ}\bQ_p$-conjugate to 
$\Lambda \otimes _{\bZ}\bZ_p$ for all $p$. \\ 
(iii)\ The constant $c_{k,j}$ for $k,j$ are defined by 
\begin{align*} 
c_{k,j} := 2^{-6}\pi ^{-3}(k-2)(j+k-1)(j+2k-3). 
\end{align*} 
(iv)\ We define $J_0'(g)$ for each $\{ g\} _G$ as follows. 
We put 
\begin{align*} 
H_g^{k,j}(Z) &:= 
\mathrm{tr}\left[ \rho_{k,j}(CZ+D)^{-1}\rho_{k,j}\left( \frac{g\langle Z\rangle -\overline{Z}}{2i}\right) ^{-1}\rho_{k,j}(Y)\right], \\ 
J_0(g) &:= 
\int _{C_0(g;Sp(2;\bR ))\backslash \fH _2} H_g^{k,j}(\hat{Z})d\hat{Z},  
\end{align*} 
where $dZ=\det (Y)^{-3}dXdY$, $dX=dx_1dx_{12}dx_2$, $dY=dy_1dy_{12}dy_2$  
for $Z=X+iY \in \fH_2$, 
$X=\begin{pmatrix} x_1&x_{12}\\x_{12}&x_2 \end{pmatrix}$, 
$Y=\begin{pmatrix} y_1&y_{12}\\y_{12}&y_2 \end{pmatrix}$, 
and $d\hat{Z}$ is an invariant measure on $C_0(g;Sp(2;\bR ))\backslash \fH_2$ 
induced from $dZ$ and a Haar measure on $C_0(g;Sp(2;\bR ))$. 
The definition of $C_0(g;Sp(2;\bR))$ is given for each $g$ 
in each subsection \ref{subsec:H1} -- \ref{subsec:H12}. 
We define $J_0'(g):=J_0(g)$ if $-1_4\not\in C_0(g;Sp(2;\bR ))$, and 
$J'_0(g):=2^{-1}\cdot J_0(g)$ if $-1_4\in C_0(g;Sp(2;\bR ))$. \\ 
(v)\ We define $M_G(\Lambda)$ for each $\{ g\} _G$ and $L_G(\Lambda)$ as follows. 
We decompose the group $(Z(g)^{\times}\cap G)_{\mathbb{A}}$ into 
the disjoint union 
\begin{align*} 
(Z(g)^{\times}\cap G)_{\mathbb{A}}=\sqcup _{i=1}^h 
(Z(g)^{\times}\cap G) y_i(\Lambda_{\mathbb{A}}^{\times} \cap G_{\mathbb{A}}), 
\end{align*} 
where $\Lambda_{\mathbb{A}}=\Lambda \otimes _{\bZ}\bZ _{\mathbb{A}}$. 
We put $\Lambda _i=y_i\Lambda y_i^{-1}=\cap _p((y_i)_p\Lambda _p(y_i)_p^{-1}\cap Z(g))$ 
and define 
\begin{align*} 
M_G(\Lambda) := \sum _{i=1}^h \mathrm{vol}\big( (\Lambda _i^{\times}\cap G)\backslash C_0(g;Sp(2;\bR ))\big). 
\end{align*} 
(vi)\ We define $c_p(g,R_p,\Lambda_p)$ for each $\{ g\} _G$, $L_G(\Lambda)$ and $p$ as 
\begin{align*} 
c_p(g,R_p,\Lambda _p)=\sharp \big( (Z(g)^{\times}\cap G)_p 
\backslash M_p(g,R_p,\Lambda _p)/(R_p^{\times}\cap G_p)\big) , 
\end{align*} 
where 
\begin{align*} 
R_p:=M(2;\fO _p)\mbox{ if }p\nmid  D_2\mbox{ and }
R_p:=\begin{pmatrix} \fO _p&\pi ^{-1}\fO _p\\ \pi \fO _p&\fO _p \end{pmatrix} 
\mbox{ if }p\mid D_2, 
\end{align*} 
\begin{align*} 
M_p(g,R_p,\Lambda _p) := \left\{ x\in G_p \left| 
\begin{array}{l} 
x^{-1}gx \in R_p,\mbox{ and } \\ 
Z(g)_p\cap xR_px^{-1} \mbox{ is } (Z(g)^{\times}\cap G)_p\mbox{-conjugate to } 
\Lambda _p 
\end{array} 
\right. \right\} . 
\end{align*} 
If $M_p(g,R_p,\Lambda_p)=\emptyset$, then we put $c_p(g,R_p,\Lambda_p)=0$. 
\ \\ 
\ 

\begin{remark} \label{rem:elliptic} 
We give some remarks about Theorem \ref{thm:elliptic}. \\ 
(1)The crucial work in using Theorem \ref{thm:elliptic} is to evaluate 
$c_p(g,R_p,\Lambda_p)$. 
This work has been completed by Hashimoto and Ibukiyama in \cite{HI80} and \cite{HI83}, so we have only to combine the deta depending on the cases. \\ 
(2)We need to determine $G$-conjugacy classes which appear in the firsr sum 
in Theorem \ref{thm:elliptic}. 
It is known that $\{g\}_G \cap \Gamma \not= \emptyset$ if and only if 
$\{g\}_{G_p} \cap R_p \not= \emptyset$ for all $p$. 
(cf. Theorem 1-3 in \cite{Has83}). 
We can obtain the result by using \cite[section 2]{HI80},\cite{Has84} and 
the results of $c_p$ mentioned above. \\ 
(3)The integral $J_0'(g)$ depends only on $Sp(2;\bR)$-conjugacy classes. 
Langlands \cite{Lan63} gave a formula for $J_0(g)$. 
We can evaluate $J_0'(g)$ by applying  
explicit formulae in \cite{Has83} and \cite{Wak}. 
\end{remark} 

We will evaluate each $H_i$ in subsection \ref{subsec:H1} -- \ref{subsec:H12}. 
We denote by $G[f_i]$ the set of torsion elements of $G$ whose principal polynomials 
are $f_i(x)$. 
For $i=1,3,4,5,7,8,10$, we have only to evaluate the contribution of $G[f_i]$ 
and double it to obtain $H_i$ 
because the contribution of $g$ is equal to that of $-g$. 
We use the notation: \\   
\[ \alpha(\theta_1,\theta_2)= 
\begin{pmatrix} 
\cos\theta_1 & 0 & \sin\theta_1 & 0 \\ 
0 & \cos\theta_2 & 0& \sin\theta_2  \\ 
-\sin\theta_1 & 0 & \cos\theta_1 & 0 \\ 
0 & -\sin\theta_2 & 0& \cos\theta_2  
\end{pmatrix}. \]

%%%%%%%%%%%%%%%%%%%%%%%%%%%%%%%%%%%%%%%%%%%%%%%%%%%%%%%%%%%%%%%%%%%%%%%%%%
%
\subsection{The contribution $\mathbf{H_1}$} \label{subsec:H1} 
%
%%%%%%%%%%%%%%%%%%%%%%%%%%%%%%%%%%%%%%%%%%%%%%%%%%%%%%%%%%%%%%%%%%%%%%%%%%

In this subsection, we consider the contribution of $\pm 1_4$. 
If $\gamma=\pm 1_4$, we have 
$C(\gamma ;Sp(2;\bR))=C_0(\gamma ;Sp(2;\bR))=Sp(2;\bR )$, 
$C(\gamma ;\Gamma )=C_0(\gamma ;\Gamma )=\Gamma$ and 
$H_{\gamma}^{k,j}(Z)=1$, 
so $J_0'(\gamma)=\frac{1}{2}\int _{Sp(2;\bR)\backslash\fH_2}d\hat{Z}$. 
Also, we have 
\begin{align*} 
c_p(\gamma ,R_p,\Lambda _p)=\left\{ \begin{array}{l} 
1\cdots \mbox{ if }\Lambda _p\sim R_p \\ 
0\cdots \mbox{ otherwise.} 
\end{array} \right. 
\end{align*}  
Hence from Theorem \ref{thm:elliptic} we have 
\begin{align*} 
H_1=2^{-6}\pi ^{-3}(k-2)(j+k-1)(j+2k-3) 
\cdot \mathrm{vol}(Sp(2;\bR )\backslash \fH _2)
\cdot \mathrm{vol}(\Gamma\backslash Sp(2;\bR )). 
\end{align*} 
We have only to multiple the value $H_1$ in the case of $D_2=1$ in \cite{Wak} 
by $\prod _{p \mid D_2} \frac{p+1}{p^2+1}$ 
because we have the indexes as follows: 
\begin{align*} 
\left[ G_p\cap \begin{pmatrix} \fO _p & \fO _p \\ \fO _p & \fO _p \end{pmatrix}^{\times} : G_p\cap \begin{pmatrix} \fO _p&\fO _p\\ \pi\fO _p&\fO _p \end{pmatrix}^{\times} \right] &= p+1, \\  
\left[ G_p\cap \begin{pmatrix} \fO _p & \pi ^{-1}\fO _p \\ \pi \fO _p & \fO _p \end{pmatrix}^{\times} : G_p\cap \begin{pmatrix} \fO _p&\fO _p\\ \pi\fO _p&\fO _p \end{pmatrix}^{\times} \right] &= p^2+1. 
\end{align*} 
Hence we obtain $H_1$ as in Theorem \ref{thm:main}.

%%%%%%%%%%%%%%%%%%%%%%%%%%%%%%%%%%%%%%%%%%%%%%%%%%%%%%%%%%%%%%%%%%%%%%%%%%
%
\subsection{The contribution $\mathbf{H_2}$} 
%
%%%%%%%%%%%%%%%%%%%%%%%%%%%%%%%%%%%%%%%%%%%%%%%%%%%%%%%%%%%%%%%%%%%%%%%%%%

In this subsection, we evaluate the contribution of $G[f_2]$, 
where $f_2(x)=(x-1)^2(x+1)^2$. 
The set $G[f_2]$ consists of 
only one $G$-conjugacy class represented by an element $g$.  
We have $Z(g)\simeq B\oplus B$. 
We fix $g$ and this isomorphism until the end of this subsection. 
We put 
\[ L:=\{ (x,y)\in \fO\oplus\fO \ |\ x-y \in \pi \fO_2 \}, \]  
where $\pi$ is a prime element of $\fO_2$. 
We have the following proposition. 
\begin{proposition} \label{prop:H2} \ \\ 
$(1)$\ The class $\{ g\} _G$ appears in the first sum of Theorem \ref{thm:elliptic}, 
i.e. $\{ g\} _G\cap \Gamma\not=\emptyset$, 
if and only if $D_2=1$ or $2$. \\ 
$(2)$\ We assume $D_2=2$. 
If $\Lambda$ is a $\bZ$-order of $Z(g)$ belonging to the same $G$-genus as $L$, 
then we have $\prod_p c_p(g,R_p,\Lambda _p)=1$.  
If $\Lambda$ does not belong to the same $G$-genus as $L$, 
then $\prod_p c_p(g,R_p,\Lambda_p)=0$. 
\end{proposition} 
\begin{proof} 
We can prove (1) and the latter part of (2) easily by \cite[Proposition 2.4]{HI83}. 
If $D_2=2$ and $\Lambda$ is a $\bZ$-order of $Z(g)$ belonging to the same $G$-genus 
as $L$, 
then it follows from 
\cite[Proposition 13]{HI80} and \cite[Proposition 2.4]{HI83} 
that $c_p(g,R_p,\Lambda_p)=1$ for any $p$. 
\end{proof} 
From Proposition \ref{prop:H2}, we have $H_2=0$ if $D_2\not= 1,2$. 
In the case where $D_2=1$, $H_2$ has been evaluated in \cite{Has84} and \cite{Wak}. 
Hereafter, we assume $D_2=2$. 
From Proposition \ref{prop:H2}, we have 
\[ H_2=c_{k,j}\cdot J_0'(g)\cdot M_G(L). \]  
We see that $g$ is $Sp(2;\bR)$-conjugate to $\alpha(\pi,0)$ and 
\[ C_0(\alpha (\pi ,0);Sp(2;\bR )) = 
\left\{ 
\left. 
\begin{pmatrix} a&0&b&0\\0&a'&0&b'\\c&0&d&0\\0&c'&0&d' \end{pmatrix} 
\right| 
\begin{array}{l} 
ad-bc=1\\a'd'-b'c'=1 
\end{array} 
\right\} , 
\] 
\[ 
J_0'(g)=\frac{1}{2}J_0(\alpha (\pi ,0)) 
={c_{k,j}}^{-1}2^{-7}\pi ^{-4}(-1)^k(j+k-1)(k-2) 
\] 
if $j$ is even. (cf. (b-5) in \cite{Wak}). 
If we put $L_0=\fO\oplus\fO$, then we have 
\begin{align*} 
M_G(L) &= \mathrm{vol}((L^{\times}\cap G)\backslash C(g;Sp(2;\bR ))) \\ 
       &= [L_0^{\times} \cap G : L^{\times} \cap G] \cdot \mathrm{vol} ((L_0^{\times}\cap G)\backslash C(g;Sp(2;\bR ))) \\ 
       &= 3^{-1}\pi ^4\prod _{p\mid D} (p-1)^2. 
\end{align*} 
(cf. (3.6),(3.7) of \cite{Has84}). 
Hence we can obtain $H_2$ as in Theorem \ref{thm:main}.

%%%%%%%%%%%%%%%%%%%%%%%%%%%%%%%%%%%%%%%%%%%%%%%%%%%%%%%%%%%%%%%%%%%%%%%%%%
%
\subsection{The contribution $\mathbf{H_3}$} 
%
%%%%%%%%%%%%%%%%%%%%%%%%%%%%%%%%%%%%%%%%%%%%%%%%%%%%%%%%%%%%%%%%%%%%%%%%%%

In this subsection, we evaluate the contribution of $G[f_3]$, 
where $f_3(x)=(x-1)^2(x^2+1)$. 
We have only to double it to obtain $H_3$. 
Note that $G[f_3]\neq\emptyset$ if and only if 
$\left( \frac{-1}{p}\right) \neq1$ for any prime divisor $p$ of $D$. 
Hereafter, we assume that $G[f_3]\neq\emptyset$. 
The set $G[f_3]$ consists of two $G$-conjugacy classes represented by 
$g$ and $g^{-1}$ for an element $g$. 
We have $Z(g)\simeq B\oplus F$ with $F=\bQ(\sqrt{-1})$. 
We fix $g$ and this isomorphism until the end of this subsection. 
We put 
\[ L:=\{ (x,y)\in \fO\oplus\cO \ |\ x-y \in \pi {\cO}_2 \}, \]  
where $\cO$ is the ring of integers of $F$ and 
$\pi$ is a prime element of ${\cO}_2$. 
Then we have the following proposition. 
\begin{proposition} \label{prop:H3} \ \\ 
$(1)$\ The classes $\{ g\} _G$ and $\{ g^{-1}\} _G$ appear 
in the first sum of Theorem \ref{thm:elliptic} 
if and only if $D_2=1$ or $2$. \\ 
$(2)$\ We assume $D_2=2$. 
If $\Lambda$ is a $\bZ$-order of $Z(g)$ belonging to the same $G$-genus as $L$, 
then we have 
\[ \prod _p c_p(g,R_p,\Lambda_p)=\prod _p c_p(g^{-1},R_p,\Lambda_p)
=2^{\sharp D_1(3;4)}. \]  
If $\Lambda$ does not belong to the same $G$-genus as $L$, 
then $\prod_p c_p(g,R_p,\Lambda_p)=0$. 
\end{proposition} 
\begin{proof} 
We can prove (1) and the latter part of (2) easily by \cite[Proposition 2.4]{HI83}. 
If $D_2=2$ and $\Lambda$ is a $\bZ$-order of $Z(g)$ belonging to the same $G$-genus 
as $L$, 
then it follows from 
\cite[Proposition 14]{HI80} and \cite[Proposition 2.4]{HI83} that 
\begin{align*} 
c_p(g,R_p,\Lambda _p)= c_p(g^{-1},R_p,\Lambda _p)= 
\begin{cases} 
2 \cdots \mbox{ if } p\mid D_1 \mbox{ and }\left(\frac{-1}{p}\right)=-1 \\ 
1 \cdots \mbox{ otherwise. } 
\end{cases} 
\end{align*} 
\end{proof}  
From Proposition \ref{prop:H3}, we have $H_3=0$ if $D_2\not= 1,2$. 
In the case where $D_2=1$, $H_3$ has been evaluated in \cite{Has84} and \cite{Wak}. 
Hereafter, we assume $D_2=2$. 
From Proposition \ref{prop:H3}, we have 
\begin{align*} 
H_3 &= 2\cdot c_{k,j}\cdot \sum _{\gamma\in\{ g,g^{-1}\}} J_0'(\gamma) \cdot 
M_{\gamma}(L)\prod _p c_p(\gamma ,R_p,L_p) \\ 
    &= 2\cdot c_{k,j}\cdot \big( J_0'(g)+J_0'(g^{-1})\big) \cdot M_g(L) \cdot 
2^{\sharp D_1(3;4)}. 
\end{align*}    
We see that $g$ and $g^{-1}$ are $Sp(2;\bR)$-conjugate to 
$\alpha(\pi/2,0)$ and $\alpha(-\pi/2,0)$ respectively and 
\[ C_0(\alpha (\pi /2 ,0),Sp(2;\bR )) = 
   C_0(\alpha (-\pi /2 ,0),Sp(2;\bR )) \] 
\[ =    
\left\{ 
\left. 
\begin{pmatrix} 1&0&0&0\\0&a&0&b\\0&0&1&0\\0&c&0&d \end{pmatrix} 
\right| 
ad-bc=1 
\right\} , 
\] 
\[ 
J_0'(g)+J_0'(g^{-1})= {c_{k,j}}^{-1}\cdot 2^4\cdot \pi ^2\cdot 
[(-1)^{\frac{j}{2}} (k-2), -(j+k-1), (-1)^{\frac{j}{2}+1} (k-2), j+k-1; 4]_k
\] 
(cf. (b-4) in \cite{Wak}). 
If we put $L_0=\fO\oplus\cO$, then we have 
\begin{align*} 
M_g(L) &= \mathrm{vol}((L^{\times}\cap G)\backslash C(g;Sp(2;\bR ))) \\ 
       &= [L_0^{\times} \cap G : L^{\times} \cap G] \cdot \mathrm{vol} ((L_0^{\times}\cap G)\backslash C(g;Sp(2;\bR ))) \\ 
       &= 2^{-2}\pi ^2\prod _{p\mid D} (p-1). 
\end{align*} 
(cf. (3.10),(3.11) of \cite{Has84}). 
Hence we can obtain $H_3$ as in Theorem \ref{thm:main}.

%%%%%%%%%%%%%%%%%%%%%%%%%%%%%%%%%%%%%%%%%%%%%%%%%%%%%%%%%%%%%%%%%%%%%%%%%%
%
\subsection{The contribution $\mathbf{H_4}$} 
%
%%%%%%%%%%%%%%%%%%%%%%%%%%%%%%%%%%%%%%%%%%%%%%%%%%%%%%%%%%%%%%%%%%%%%%%%%%

In this subsection, we evaluate the contribution of $G[f_4]$, 
where $f_4(x)=(x-1)^2(x^2+x+1)$. 
We have only to double it to obtain $H_4$. 
Note that $G[f_4]\neq\emptyset$ if and only if $\left( \frac{-3}{p}\right)\neq1$ 
for any prime divisor $p$ of $D$. 
Hereafter, we assume that $D$ does not have such a prime divisor. 
The set $G[f_4]$ consists of two $G$-conjugacy classes represented by 
$g$ and $g^{-1}$ for an element $g$. 
We have $Z(g)\simeq B\oplus F$ with $F=\bQ(\sqrt{-3})$. 
We fix $g$ and this isomorphism until the end of this subsection. 
We put 
\[ L:=\{ (x,y)\in \fO\oplus\cO \ |\ x-y \in \pi{\cO}_3 \}, \]  
where $\cO$ is the ring of integers of $F$ and 
$\pi$ is a prime element of ${\cO}_3$. 
Then we have the following proposition. 
\begin{proposition} \label{prop:H4} \ \\ 
$(1)$\ The classes $\{ g\} _G$ and $\{ g^{-1}\} _G$ appear 
in the first sum of Theorem \ref{thm:elliptic} 
if and only if $D_2=1$ or $3$. \\ 
$(2)$\ We assume $D_2=3$. 
If $\Lambda$ is a $\bZ$-order of $Z(g)$ belonging to the same $G$-genus as $L$, 
then we have 
\[ \prod _p c_p(g,R_p,\Lambda_p)=\prod _p c_p(g^{-1},R_p,\Lambda_p)
=2^{\sharp D_1(2;3)}. \]  
If $\Lambda$ does not belong to the same $G$-genus as $L$, 
then $\prod_p c_p(g,R_p,\Lambda_p)=0$. 
\end{proposition} 
\begin{proof} 
We can prove (1) and the latter part of (2) easily by \cite[Proposition 2.4]{HI83}. 
If $D_2=3$ and $\Lambda$ is a $\bZ$-order of $Z(g)$ belonging to the same $G$-genus 
as $L$, 
then it follows from 
\cite[Proposition 14]{HI80} and \cite[Proposition 2.4]{HI83} that 
\begin{align*} 
c_p(g,R_p,\Lambda _p)= c_p(g^{-1},R_p,\Lambda _p)= 
\begin{cases} 
2 \cdots \mbox{ if } p\mid D_1 \mbox{ and }\left(\frac{-3}{p}\right)=-1 \\ 
1 \cdots \mbox{ otherwise. } 
\end{cases} 
\end{align*} 
\end{proof} 
From Proposition \ref{prop:H4}, we have 
\begin{align*} 
H_4 &= 2\cdot c_{k,j}\cdot \sum _{\gamma\in\{ g,g^{-1}\}} J_0'(\gamma) \cdot 
M_{\gamma}(L)\prod _p c_p(\gamma ,R_p,L_p) \\ 
    &= 2\cdot c_{k,j}\cdot \big( J_0'(g)+J_0'(g^{-1})\big) \cdot M_g(L) \cdot 
2^{\sharp D_1(2;3)}. 
\end{align*}    
We see that $g$ and $g^{-1}$ are $Sp(2;\bR)$-conjugate to 
$\alpha(2\pi/3,0)$ and $\alpha(-2\pi/3,0)$ respectively and 
\[ C_0(\alpha (2\pi /3,0),Sp(2;\bR )) = 
   C_0(\alpha (-2\pi /3,0),Sp(2;\bR )) \] 
\[ =    
\left\{ 
\left. 
\begin{pmatrix} 1&0&0&0\\0&a&0&b\\0&0&1&0\\0&c&0&d \end{pmatrix} 
\right| 
ad-bc=1 
\right\} , 
\] 
\begin{align*} 
J_0'(g)+J_0'(g^{-1}) &= {c_{k,j}}^{-1}\cdot 2^{-3}\cdot 3^{-1}\cdot \pi ^{-2} \\  
&\hspace{7mm} \times \big\{ [j+k-1, -(j+k-1), 0; 3]_k +[k-2, 0, -(k-2); 3]_{j+k}\big\} 
\end{align*} 
(cf. (b-4) in \cite{Wak}). 
If we put $L_0=\fO\oplus\cO$, then we have 
\begin{align*} 
M_g(L) &= \mathrm{vol}((L^{\times}\cap G)\backslash C(g;Sp(2;\bR ))) \\ 
       &= [L_0^{\times} \cap G : L^{\times} \cap G] \cdot \mathrm{vol} ((L_0^{\times}\cap G)\backslash C(g;Sp(2;\bR ))) \\ 
       &= 2^2\cdot 3^{-2}\cdot \pi ^2\prod _{p\mid D} (p-1). 
\end{align*} 
(cf. (3.10),(3.11) of \cite{Has84}). 
Hence we can obtain $H_4$ as in Theorem \ref{thm:main}.

%%%%%%%%%%%%%%%%%%%%%%%%%%%%%%%%%%%%%%%%%%%%%%%%%%%%%%%%%%%%%%%%%%%%%%%%%%
%
\subsection{The contribution $\mathbf{H_5}$} 
%
%%%%%%%%%%%%%%%%%%%%%%%%%%%%%%%%%%%%%%%%%%%%%%%%%%%%%%%%%%%%%%%%%%%%%%%%%%

In this subsection, we consider the contribution of $G[f_5]$, 
where $f_5(x)=(x-1)^2(x^2-x+1)$. 
We have only to double it to obtain $H_5$. 
The set $G[f_5]$ consists of two $G$-conjugacy classes represented by 
$g$ and $g^{-1}$ for an element $g$. 
We see from \cite[Proposition 14]{HI80} and \cite[Proposition 2.4]{HI83} that 
the classes $\{ g\} _G$ and $\{ g^{-1}\} _G$ do not appear 
in the first sum of Theorem \ref{thm:elliptic} if $D_2\not= 1$. 
In the case where $D_2=1$, $H_5$ has been evaluated in \cite{Has84} and \cite{Wak}.

%%%%%%%%%%%%%%%%%%%%%%%%%%%%%%%%%%%%%%%%%%%%%%%%%%%%%%%%%%%%%%%%%%%%%%%%%%
%
\subsection{The contribution $\mathbf{H_6}$} 
%
%%%%%%%%%%%%%%%%%%%%%%%%%%%%%%%%%%%%%%%%%%%%%%%%%%%%%%%%%%%%%%%%%%%%%%%%%%

In this subsection, we consider the contribution of $G[f_6]$, 
where $f_6(x)=(x^2+1)^2$. 
Note that $G[f_6]=\emptyset$ if and only if $D$ has a prime divisor $p$ 
with $\left(\frac{-1}{p}\right) =1$.  
Hereafter, we assume that $D$ does not have such a prime divisor. 
Then there are infinitely many $G$-conjugacy classes in $G[f_6]$. 
As in \cite[Theorem 3.2 (i),(ii)]{Has84}, 
we have a correspondence between the set of $G$-conjugacy classes $\{ g\} _G$'s 
in $G[f_6]$ and the set of isomorphism classes of quaternion algebras $Z_0(g)$'s 
over $\bQ$ which are contained in $B\otimes_{\bQ}F$, with $F=\bQ (\sqrt{-1})$. 
This correspondence is two-to-one or one-to-one according as $Z_0(g)$ is definite 
of indefinite. 
We denote by $D(Z_0(g))$ the discriminant of $Z_0(g)$. 
If $Z_0(g)$ is definite, 
two $G$-conjugacy classes $\{ g\}_G$ and $\{ g^{-1}\}_G$ correspond to $Z_0(g)$. 
In this case, $g$ is $Sp(2;\bR)$-conjugate to 
$\alpha (\tfrac{\pi}{2},\tfrac{\pi}{2})$ and 
\[ J_0'(g)+J_0'(g^{-1}) 
={c_{k,j}}^{-1}\cdot 2^{-2}\cdot (j+1)\cdot (-1)^{k+j/2}, \] 
if $j$ is even (cf. (b-2) in \cite{Wak}), 
\[ M_G(\Lambda )=\frac{1}{48} \prod_{p\mid D(Z_0(g))}(p-1)\prod_{p}d_p(\Lambda )/e_p(\Lambda ) \] 
for a $\bZ$-order $\Lambda$ of $Z(g)$ (cf. \cite[Proposition 12]{HI80}), where 
$\fO _0$ is a maximal order of $Z_0(g)$ and 
\begin{align*} 
d_p(\Lambda ) &= [ {\fO _0}_p^{\times} : (\Lambda \cap Z_0(g))_p^{\times} ], &  
e_p(\Lambda ) &= [ \Lambda_p^{\times}\cap G_p : {\fO _0}_p^{\times}\cdot {\cO_F}_p^{\times} ]. 
\end{align*} 
On the other hand, 
if $Z_0(g)$ is indefinite, then 
only one $G$-conjugacy class $\{ g\}_G$ corresponds to $Z_0(g)$. 
In this case, 
$g$ is $Sp(2;\bR)$-conjugate to 
$\alpha (\tfrac{\pi}{2},-\tfrac{\pi}{2})$ and 
\[ J_0'(g)
={c_{k,j}}^{-1}\cdot 2^{-5}\cdot \pi ^{-2}\cdot (j+2k-3)\cdot (-1)^{j/2}, \] 
if $j$ is even (cf. (b-3) in \cite{Wak}), 
\[ M_G(\Lambda )=\frac{\pi ^2}{6} \prod_{p\mid D(Z_0(g))}(p-1)\prod_{p}d_p(\Lambda )/e_p(\Lambda ) \] 
for a $\bZ$-order $\Lambda$ of $Z(g)$ (cf. \cite[(3.16)]{Has84}). 
\begin{proposition} \label{prop:H6} 
(1)\ The class $\{ g\}_G$ appears in the first sum of Theorem \ref{thm:elliptic} 
if and only if $D(Z_0(g))\mid 2D_1$. \\ 
(2)\ (i)\ the case where $2\mid D_1$, \\ 
$\cdot$ If a $G$-conjugacy class $\{ g\}_G$ satisfies $2\mid D(Z_0(g))$, 
then two $G$-genus of $\bZ$-orders of $Z(g)$ appear in the second sum of 
Theorem \ref{thm:elliptic}, and 
\[ \prod _p d_p(\Lambda )/e_p(\Lambda )=
\prod _{p\nmid D(Z_0(g)) \atop  p\mid D_2,\ p\neq 2}(p+1) 
\cdot \ 3/2 \mathrm{\ (resp.\ 1)}  \] 
$\cdot$ If a $G$-conjugacy class $\{ g\}_G$ satisfies $2\nmid D(Z_0(g))$, 
then three $G$-genus of $\bZ$-orders of $Z(g)$ appear in the second sum of 
Theorem \ref{thm:elliptic}, and 
\[ \prod _p d_p(\Lambda )/e_p(\Lambda )=
\prod _{p\nmid D(Z_0(g)) \atop  p\mid D_2,\ p\neq 2}(p+1) 
\cdot \ 3/2 \mathrm{\ (resp.\ 1\ and\ 3)}  \] 
(ii)\ the case where $2\mid D_2$, \\ 
$\cdot$ If a $G$-conjugacy class $\{ g\}_G$ satisfies $2\mid D(Z_0(g))$, 
then two $G$-genus of $\bZ$-orders of $Z(g)$ appear in the second sum of 
Theorem \ref{thm:elliptic}, and 
\[ \prod _p d_p(\Lambda )/e_p(\Lambda )=
\prod _{p\nmid D(Z_0(g)) \atop  p\mid D_2,\ p\neq 2}(p+1) 
\cdot \ 1/2 \mathrm{\ (resp.\ 3)}  \] 
$\cdot$ If a $G$-conjugacy class $\{ g\}_G$ satisfies $2\nmid D(Z_0(g))$, 
then two $G$-genus of $\bZ$-orders of $Z(g)$ appear in the second sum of 
Theorem \ref{thm:elliptic}, and 
\[ \prod _p d_p(\Lambda )/e_p(\Lambda )=
\prod _{p\nmid D(Z_0(g)) \atop  p\mid D_2,\ p\neq 2}(p+1) 
\cdot \ 3 \mathrm{\ (resp.\ 3/2)}  \] 
(iii)\ the case where $2\nmid D$, \\ 
$\cdot$ If a $G$-conjugacy class $\{ g\}_G$ satisfies $2\mid D(Z_0(g))$, 
then only one $G$-genus of $\bZ$-orders of $Z(g)$ appear in the second sum of 
Theorem \ref{thm:elliptic}, and 
\[ \prod _p d_p(\Lambda )/e_p(\Lambda )=
\prod _{p\nmid D(Z_0(g)) \atop  p\mid D_2,\ p\neq 2}(p+1) 
\cdot 3/2 \] 
$\cdot$ If a $G$-conjugacy class $\{ g\}_G$ satisfies $2\nmid D(Z_0(g))$, 
then two $G$-genus of $\bZ$-orders of $Z(g)$ appear in the second sum of 
Theorem \ref{thm:elliptic}, and 
\[ \prod _p d_p(\Lambda )/e_p(\Lambda )=
\prod _{p\nmid D(Z_0(g)) \atop  p\mid D_2,\ p\neq 2}(p+1) 
\cdot \ 3/2 \mathrm{\ (resp.\ 1)}  \] 
(3)\ For any case and any $G$-genus, we have 
$\displaystyle \prod\limits _p c_p(g,R_p,\lambda_p)=\prod\limits_{p\nmid D(Z_0(g)) \atop p\mid D_1,\ p\neq 2} 2$. 
\end{proposition} 
\begin{proof} 
We see from \cite[Proposition 15 and 16]{HI80} and \cite[Proposition 2.5 and 2.6]{HI83} that \\ 
\begin{center} 
$\begin{array}{lcc}  
p\mid D_1 & \Longrightarrow & \{ g\} _{G_p} \cap R_p \neq \emptyset,\\ 
p\nmid D_1\mbox{ and }p\nmid Z_0(g) & \Longrightarrow & \{ g\} _{G_p} \cap R_p \neq \emptyset, \\ 
p\nmid D_1,\ p\mid Z_0(g)\mbox{ and }p=2 & \Longrightarrow & \{ g\} _{G_p} \cap R_p \neq \emptyset, \\ 
p\nmid D_1,\ p\mid Z_0(g)\mbox{ and }p\neq 2 & \Longrightarrow & \{ g\} _{G_p} \cap R_p =\emptyset.
\end{array}$
\end{center}  
Hence we obtain (1) (cf. \cite[Theorem 1-3]{Has80}). 
Also we can obtain (2),(3) from the above four propositions.  
\end{proof}

%%%%%%%%%%%%%%%%%%%%%%%%%%%%%%%%%%%%%%%%%%%%%%%%%%%%%%%%%%%%%%%%%%%%%%%%%%
%
\subsection{The contribution $\mathbf{H_7}$} 
%
%%%%%%%%%%%%%%%%%%%%%%%%%%%%%%%%%%%%%%%%%%%%%%%%%%%%%%%%%%%%%%%%%%%%%%%%%%

In this subsection, we consider the contribution of $G[f_7]$, 
where $f_7(x)=(x^2+x+1)^2$. 
We have only to double it to obtain $H_7$. 
Note that $G[f_7]=\emptyset$ if and only if $D$ has a prime divisor $p$ 
with $\left(\frac{-3}{p}\right) =1$.  
Hereafter, we assume that $D$ does not have such a prime divisor. 
We can use the same method as in the case of $H_6$. 
we have a correspondence between the set of $G$-conjugacy classes $\{ g\} _G$'s 
in $G[f_7]$ and the set of isomorphism classes of quaternion algebras $Z_0(g)$'s 
over $\bQ$ which are contained in $B\otimes_{\bQ}F$, with $F=\bQ (\sqrt{-3})$. 
If $Z_0(g)$ is definite, then 
\[ J_0'(g)+J_0'(g^{-1}) 
={c_{k,j}}^{-1}3^{-1}(j+1)[0,1,-1;3]_{j+2k}, \] 
if $j$ is even (cf. (b-2) in \cite{Wak}), and 
\[ M_G(\Lambda )=\frac{1}{72} \prod_{p\mid D(Z_0(g))}(p-1)\prod_{p}d_p(\Lambda )/e_p(\Lambda ) \] 
for a $\bZ$-order $\Lambda$ of $Z(g)$.  
If $Z_0(g)$ is indefinite, then 
\[ J_0'(g)={c_{k,j}}^{-1}2^3 \cdot 3\pi ^2 (j+2k-3)[1,-1,0;3]_j, \] 
if $j$ is even (cf. (b-3) in \cite{Wak}), and 
\[ M_G(\Lambda )=\frac{\pi ^2}{3^2} \prod_{p\mid D(Z_0(g))}(p-1)\prod_{p}d_p(\Lambda )/e_p(\Lambda ). \] 
\begin{proposition} 
(1)\ The class $\{ g\}_G$ appears in the first sum of Theorem \ref{thm:elliptic} 
if and only if $D(Z_0(g))\mid 3D_1$. \\ 
(2)\ (i)\ the case where $3\mid D_1$, \\ 
$\cdot$ If a $G$-conjugacy class $\{ g\}_G$ satisfies $3\mid D(Z_0(g))$, 
then only one $G$-genus of $\bZ$-orders of $Z(g)$ appears in the second sum of 
Theorem \ref{thm:elliptic}, and 
\[ \prod _p d_p(\Lambda )/e_p(\Lambda )=
\prod _{p\nmid D(Z_0(g)) \atop  p\mid D_2,\ p\neq 3}(p+1) 
\cdot \ 1/2   \] 
$\cdot$ If a $G$-conjugacy class $\{ g\}_G$ satisfies $3\nmid D(Z_0(g))$, 
then two $G$-genus of $\bZ$-orders of $Z(g)$ appear in the second sum of 
Theorem \ref{thm:elliptic}, and 
\[ \prod _p d_p(\Lambda )/e_p(\Lambda )=
\prod _{p\nmid D(Z_0(g)) \atop  p\mid D_2,\ p\neq 3}(p+1) 
\cdot \ 2 \mathrm{\ (resp.\ 6)}  \] 
(ii)\ the case where $3\mid D_2$, \\ 
$\cdot$ If a $G$-conjugacy class $\{ g\}_G$ satisfies $3\mid D(Z_0(g))$, 
then only one  $G$-genus of $\bZ$-orders of $Z(g)$ appears in the second sum of 
Theorem \ref{thm:elliptic}, and 
\[ \prod _p d_p(\Lambda )/e_p(\Lambda )=
\prod _{p\nmid D(Z_0(g)) \atop  p\mid D_2,\ p\neq 3}(p+1) 
\cdot \ 1  \] 
$\cdot$ If a $G$-conjugacy class $\{ g\}_G$ satisfies $3\nmid D(Z_0(g))$, 
then two $G$-genus of $\bZ$-orders of $Z(g)$ appear in the second sum of 
Theorem \ref{thm:elliptic}, and 
\[ \prod _p d_p(\Lambda )/e_p(\Lambda )=
\prod _{p\nmid D(Z_0(g)) \atop  p\mid D_2,\ p\neq 3}(p+1) 
\cdot \ 1 \mathrm{\ (resp.\ 4)}  \] 
(iii)\ the case where $3\nmid D$, \\ 
$\cdot$ If a $G$-conjugacy class $\{ g\}_G$ satisfies $3\mid D(Z_0(g))$, 
then only one $G$-genus of $\bZ$-orders of $Z(g)$ appears in the second sum of 
Theorem \ref{thm:elliptic}, and 
\[ \prod _p d_p(\Lambda )/e_p(\Lambda )=
\prod _{p\nmid D(Z_0(g)) \atop  p\mid D_2,\ p\neq 3}(p+1) 
\cdot 1/2 \] 
$\cdot$ If a $G$-conjugacy class $\{ g\}_G$ satisfies $3\nmid D(Z_0(g))$, 
then only one  $G$-genus of $\bZ$-orders of $Z(g)$ appears in the second sum of 
Theorem \ref{thm:elliptic}, and 
\[ \prod _p d_p(\Lambda )/e_p(\Lambda )=
\prod _{p\nmid D(Z_0(g)) \atop  p\mid D_2,\ p\neq 3}(p+1) 
\cdot \ 1  \] 
(3)\ For any case and any $G$-genus, we have 
$\displaystyle \prod\limits _p c_p(g,R_p,\lambda_p)=\prod\limits_{p\nmid D(Z_0(g)) \atop p\mid D_1,\ p\neq 3} 2$. 
\end{proposition} 
\begin{proof} 
We can prove this by the same way as Proposition \ref{prop:H6}. 
\end{proof}

%%%%%%%%%%%%%%%%%%%%%%%%%%%%%%%%%%%%%%%%%%%%%%%%%%%%%%%%%%%%%%%%%%%%%%%%%%
%
\subsection{The contribution $\mathbf{H_8}$} 
%
%%%%%%%%%%%%%%%%%%%%%%%%%%%%%%%%%%%%%%%%%%%%%%%%%%%%%%%%%%%%%%%%%%%%%%%%%%

In this subsection, we evaluate the contribution of $G[f_8]$, 
where $f_8(x)=(x^2+1)(x^2+x+1)$. 
We have only to double it to obtain $H_8$. 
We see from \cite[Proposition 2.7]{HI83} that 
no $G$-conjugacy classes corresponding to $f_8(\pm x)$ 
appear in Theorem \ref{thm:elliptic} if $D_2\not= 1$. 
In the cases where $D_2=1$, $H_8$ has been evaluated in \cite{Has84} and \cite{Wak}.

%%%%%%%%%%%%%%%%%%%%%%%%%%%%%%%%%%%%%%%%%%%%%%%%%%%%%%%%%%%%%%%%%%%%%%%%%%
%
\subsection{The contribution $\mathbf{H_9}$} 
%
%%%%%%%%%%%%%%%%%%%%%%%%%%%%%%%%%%%%%%%%%%%%%%%%%%%%%%%%%%%%%%%%%%%%%%%%%%

In this subsection, we evaluate the contribution of $G[f_9]$,   
where $f_{9}(x)=(x^2+x+1)(x^2-x+1)$. 
Note that $G[f_9]\neq\emptyset$ if and only if 
$\left( \frac{-3}{p}\right)\neq 1$ for any prime divisor $p$ of $D$.  
Hereafter, we assume that $G[f_9]\neq\emptyset$. 
We put $F:=\bQ (\sqrt{-3})$, then $Z(g)\simeq F\oplus F$ for any $g$. 
We put 
\[ L:=\{ (x,y)\in\cO\oplus\cO \mid x-y\in 2\cO \} , \]  
where $\cO$ is the ring of integers of $F$, then 
we have the following proposition. 
\begin{proposition} \label{prop:H9} \ \\ 
$(1)$\ If $D_2\neq 1,2$, then no $G$-conjugacy classes in $G[f_9]$ appear in 
the first sum of Theorem \ref{thm:elliptic}. \\ 
$(2)$\ If $D_2=2$, then the followings hold. 

$(i)$\ The number of $G$-conjugacy classes in $G[f_9]$ 
which appear in the first sum of 
Theorem \ref{thm:elliptic} is $4\cdot 2^{\sharp D_1(2;3)}$. 

$(ii)$\ Let $\{ g\}_G$ be any one of them. 
If $\Lambda$ is a $\bZ$-order of $Z(g)$ belonging to the same $G$-genus as $L$, 
then 
\[ \prod_p c_p(g,R_p,\Lambda_p)=2\cdot 2^{\sharp D_1(2;3)}. \]  
If $\Lambda$ does not belong to the same $G$-genus as $L$, 
then $\prod_p c_p(g,R_p,\Lambda_p)=0$. 
\end{proposition} 
\begin{proof} 
We can obtain (1) and the latter part of (2\ ii) from \cite[Proposition 2.7]{HI83}. 
For an element $g$ of $G[f_9]$, $g$ is $G_p$-conjugate to 
\[ \left\{ \begin{array}{lcl} 
\gamma_p & \cdots & \mbox{ if } \left( \frac{-3}{p}\right) = 1 \\ 
\gamma_p \mbox{ or } \delta_p & \cdots & \mbox{ if } \left( \frac{-3}{p}\right) \neq 1 
\end{array} \right. \] 
for some elements $\gamma_p$ and $\delta_p \in G_p$. 
It follows from \cite[Proposition 18]{HI80} and \cite[Proposition 2.7]{HI83} that 
if $D_2=2$ and 
$\Lambda$ is a $\bZ$-order of $Z(g)$ belonging to the same $G$-genus as $L$, 
then $c_p$ is as in the following table 
for each prime number $p$ satisfying the first column:  
\begin{center} 
\begin{tabular}{c|c|c} 
$p$ & $c_p(\gamma_p,R_p,\Lambda_p)$ & $c_p(\delta_p,R_p,\Lambda_p)$ \\ \hline 
$p\mid D_1$ and $\left( \frac{-3}{p}\right)=-1$ & $2$ & $2$ \\ 
$p\nmid D$ and $\left( \frac{-3}{p}\right)=1$ & $1$ & $\times$ \\ 
$p\nmid D$ and $\left( \frac{-3}{p}\right)=-1$ & $1$ & $0$ \\ 
$p=2$ & $2$ & $0$ \\ 
$p=3$ & $1$ & $1$ 
\end{tabular} 
\end{center} 
Also, $g$ is $Sp(2;\bR)$-conjugate to 
$g_1:=\alpha (\pi /3,2\pi /3)$, ${g_1}^{-1}=\alpha (-\pi /3,-2\pi /3)$, 
$g_2:=\alpha (\pi /3,-2\pi /3)$, and ${g_2}^{-1}=\alpha (-\pi /3,2\pi /3)$. 
Since $g^2$ belongs to $G[f_7]$, $g$ is $Sp(2;\bR)$-conjugate to $g_1$ or $g_1^{-1}$ 
if $Z_0(g^2)$ is indefinite, 
and $g$ is $Sp(2;\bR)$-conjugate to $g_2$ or $g_2^{-1}$ 
if $Z_0(g^2)$ is definite. 
We take all combinations of $G_p$-conjugations, 
and also we take $Sp(2;\bR)$-conjugation out of ^^ ^^ $g_1$ or $g_1^{-1}$" or 
^^ ^^ $g_2$ or $g_2^{-1}$", 
according as $Z_0(g^2)$ is indefinite or definte. 
Then $G$-conjugacy class is determined uniquely for them  by Hasse principle 
(\cite[Theorem 1-2]{Has80}). 
\end{proof} 
We see from Proposition \ref{prop:H9} that $H_9=0$ if $D_2\neq 1,2$. 
In the case where $D_2=1$, $H_2$ has been evaluated in \cite{Has84} and \cite{Wak}. 
Hereafter, we assume $D_2=2$. 
We obtain from Proposition \ref{prop:H9} that  
\[ H_9=c_{k,j}\cdot \sum _{\{ g\}_G} J_0'(g)\cdot M_G(L)\cdot \prod_p c_p(g,R_p,L_p), \] 
\[ \sum _{\{ g\}_G} J_0'(g)=
\Big( J_0'(g_1)+J_0'(g_1^{-1})+J_0'(g_2)+J_0'(g_2^{-1})\Big) 
\cdot 2^{\sharp D_1(2;3)}. \] 
We have $C_0(g;Sp(2;\bR))=\{ 1_4\}$ for any $g$, and \\ 
\hspace{20pt} $J_0'(g_1)+J_0'({g_1}^{-1})+J_0'(g_2)+J_0'({g_2}^{-1})$ \\ 
\hspace{110pt} $=c_{k,j}^{-1}\cdot 
\left\{ \begin{array}{ccl} 
{[1,0,0,-1,0,0;6]_k}&\cdots &\mbox{if }j\equiv 0 \mbox{ mod } 6 \\  
{[-1,1,0,1,-1,0;6]_k}&\cdots &\mbox{if }j\equiv 2 \mbox{ mod } 6 \\  
{[0,-1,0,0,1,0;6]_k}&\cdots &\mbox{if }j\equiv 4 \mbox{ mod } 6   
\end{array} \right.$. \\   
(cf. (b-1) in \cite{Wak}). 
We have 
\[ M_G(L)=\frac{1}{12}, \hspace{10pt} \text{(cf. (3.21) in \cite{Has84})} \] 
\[ \prod_p c_p(g,R_p,L_p)=2\cdot 2^{\sharp D_1(2;3)}\]  
for any $g$. 
Hence we can obtain $H_9$ as in Theorem \ref{thm:main}.

%%%%%%%%%%%%%%%%%%%%%%%%%%%%%%%%%%%%%%%%%%%%%%%%%%%%%%%%%%%%%%%%%%%%%%%%%%
%
\subsection{The contribution $\mathbf{H_{10}}$} 
%
%%%%%%%%%%%%%%%%%%%%%%%%%%%%%%%%%%%%%%%%%%%%%%%%%%%%%%%%%%%%%%%%%%%%%%%%%%

In this subsection, we evaluate the contribution of $G[f_{10}]$, 
where $f_{10}(x)=x^4+x^3+x^2+x+1$. 
We have only to double it to obtain $H_{10}$. 
Note that if $D(1;5)\not= \emptyset$, then $G[f_{10}]=\emptyset$. 
Hereafter, we assume $D(1;5)=\emptyset$. 
We have $Z(g)=\bQ (g)\simeq \bQ (\zeta _5)$ for any $g$. 
We have the following proposition. 
\begin{proposition} \label{prop:H10} \ \\ 
$(1)$\ If $D_1(2;5)\sqcup D_1(3;5)\sqcup D_2(4;5)\neq\emptyset$, 
then no $G$-conjugacy classes in $G[f_{10}]$ appear in 
the first sum of Theorem \ref{thm:elliptic}. \\ 
$(2)$\ If $D_1(2;5)\sqcup D_1(3;5)\sqcup D_2(4;5)=\emptyset$, 
then the followings hold. 

$(i)$\ The number of $G$-conjugacy classes in $G[f_{10}]$ 
which appear in the first sum of 
Theorem \ref{thm:elliptic} is $4\cdot 2^{\sharp D_1(4;5)}$. 

$(ii)$\ Let $\{ g\}_G$ be any one of them. 
If $\Lambda$ is a $\bZ$-order of $Z(g)$ belonging to the same $G$-genus as $\cO$, 
where $\cO$ is the ring of integers of $Z(g)$,  
then 
\[ \prod_p c_p(g,R_p,\Lambda _p)=2^{\sharp D_1(4;5)}\cdot 2^{\sharp D_2(2;5)}
\cdot 2^{\sharp D_2(3;5)}. \]  
If $\Lambda$ does not belong to the same $G$-genus as $\cO$, 
then $\prod_p c_p(g,R_p,\Lambda_p)=0$. 
\end{proposition} 
\begin{proof} 
We can obtain (1) and the latter part of (2 ii) 
from \cite[Proposition 19]{HI80} and \cite[Proposition 2.8]{HI83}. 
For any element $g$ of $G[f_{10}]$,  
$g$ is $G_p$-conjugate to 
\[ \left\{ \begin{array}{lcl} 
\gamma_p & \cdots & \mbox{ if } p\not\in D_1(4;5) \\ 
\gamma_p \mbox{ or } \delta_p & \cdots & \mbox{ if } p\in D_1(4;5) 
\end{array} \right. \] 
for some elements $\gamma_p$ and $\delta_p \in G_p$. 
It follows from \cite[Proposition 19]{HI80} and \cite[Proposition 2.8]{HI83} that 
if $D_1(2;5)\sqcup D_1(3;5)\sqcup D_2(4;5)=\emptyset$ and 
$\Lambda$ is a $\bZ$-order of $Z(g)$ belonging to the same $G$-genus as $\cO$, 
then $c_p$ is as in the following table 
for each prime number $p$ satisfying the first column: 
\begin{center} 
\begin{tabular}{c|c|c} 
$p$ &$c_p(\gamma_p,R_p,\Lambda_p)$&$c_p(\delta_p,R_p,\Lambda_p)$ \\ \hline 
$p\mid D_1$ and $p\equiv 2,3$ mod $5$&$0$&$\times$ \\ 
$p\mid D_1$ and $p\equiv 4$ mod $5$&$2$&$2$ \\ 
$p\mid D_1$ and $p=5$&$1$&$\times$ \\ 
$p\mid D_2$ and $p\equiv 2,3$ mod $5$&$2$&$\times$ \\ 
$p\mid D_2$ and $p\equiv 4$ mod $5$&$0$&$0$ \\ 
$p\mid D_2$ and $p=5$&$1$&$\times$ \\ 
$p\nmid D$ and $p\equiv 1$ mod $5$&$1$&$\times$ \\ 
$p\nmid D$ and $p\equiv 2,3$ mod $5$&$1$&$\times$ \\ 
$p\nmid D$ and $p\equiv 4$ mod $5$&$1$&$0$ \\ 
$p\nmid D$ and $p=5$&$1$&$\times$ \\ 
\end{tabular} 
\end{center} 
Also, $g$ is $Sp(2;\bR)$-conjugate to 
$g_1:=\alpha (2\pi /5,4\pi /5)$, ${g_1}^{-1}=\alpha (-2\pi /5,-4\pi /5)$, 
$g_2:=\alpha (2\pi /5,-4\pi /5)$, and ${g_2}^{-1}=\alpha (-2\pi /5,4\pi /5)$. 
We can take all combinations of $Sp(2;\bR)$-conjugation and $G_p$-conjugations. 
\end{proof} 
We see from Proposition \ref{prop:H10} that $H_{10}=0$ if 
$D_1(2;5)\sqcup D_1(3;5)\sqcup D_2(4;5)\neq\emptyset$. 
In the other cases, we obtain from Proposition \ref{prop:H10} that 
\[ H_{10}=c_{k,j}\cdot \sum _{\{ g\}_G} J_0'(g)\cdot M_G(\cO )\cdot \prod_p c_p(g,R_p,\cO _p), \] 
\[ \sum _{\{ g\}_G} J_0'(g)=
\Big( J_0'(g_1)+J_0'(g_1^{-1})+J_0'(g_2)+J_0'(g_2^{-1})\Big) 
\cdot 2^{\sharp D_1(4;5)}. \] 
We have $C_0(g;Sp(2;\bR))=\{ 1_4\}$ for any $g$, and \\ 
\scalebox{0.9}[1]{
$J_0'(g_1)+J_0'({g_1}^{-1})+J_0'(g_2)+J_0'({g_2}^{-1})=
\left\{ \begin{array}{ccl} 
{[1,0,0,-1,0;5]_k}&\cdots & \mbox{if }j\equiv 0 \mbox{ mod } 10 \\ 
{[-1,1,0,0,0;5]_k}&\cdots & \mbox{if }j\equiv 2 \mbox{ mod } 10 \\ 
0&\cdots & \mbox{if }j\equiv 4 \mbox{ mod } 10 \\ 
{[0,0,0,1,-1;5]_k}&\cdots & \mbox{if }j\equiv 6 \mbox{ mod } 10 \\ 
{[0,-1,0,0,1;5]_k}&\cdots & \mbox{if }j\equiv 8 \mbox{ mod } 10 
\end{array} \right.$} . \\   
(cf. (b-1) in \cite{Wak}). 
We have 
\[ M_G(\cO )=\frac{1}{10}, \hspace{10pt} \text{(cf. (3.23) in \cite{Has84})} \] 
\[ \prod_p c_p(g,R_p,\cO _p)=2^{\sharp D_1(4;5)}\cdot 2^{\sharp D_2(2;5)}\cdot 2^{\sharp D_2(3;5)} \]  
for any $g$. 
Hence we can obtain $H_{10}$ as in Theorem \ref{thm:main}.

%%%%%%%%%%%%%%%%%%%%%%%%%%%%%%%%%%%%%%%%%%%%%%%%%%%%%%%%%%%%%%%%%%%%%%%%%%
%
\subsection{The contribution $\mathbf{H_{11}}$} 
%
%%%%%%%%%%%%%%%%%%%%%%%%%%%%%%%%%%%%%%%%%%%%%%%%%%%%%%%%%%%%%%%%%%%%%%%%%%

In this subsection, we evaluate the contribution of $G[f_{11}]$,   
where $f_{11}(x)=x^4+1$. 
Note that if $D(1;8)\not=\emptyset$, then $G[f_{11}]=\emptyset$. 
Hereafter, we assume $D(1;8)=\emptyset$. 
We have $Z(g)=\bQ (g)\simeq \bQ (\zeta_8)$ for any $g$. 
We have the following proposition. 
\begin{proposition} \label{prop:H11} \ \\ 
$(1)$\ If $D_2(7;8)\neq\emptyset$, 
then no $G$-conjugacy classes of $G[f_{11}]$ appear in 
the first sum of Theorem \ref{thm:elliptic}. \\ 
$(2)$\ If $D_2(7;8)=\emptyset$, 
then the followings hold. 

$(i)$\ The number of $G$-conjugacy classes in $G[f_{11}]$ 
which appear in the first sum of 
Theorem \ref{thm:elliptic} is $4\cdot 2^{\sharp D_1(7;8)}$. 

$(ii)$\ Let $\{ g\}_G$ be any one of them. 
If $\Lambda$ is a $\bZ$-order of $Z(g)$ belonging to the same $G$-genus as $\cO$, 
where $\cO$ is the ring of integers of $Z(g)$,  
then 
\[ \prod_p c_p(g,R_p,\Lambda_p)=\prod _{p \mid D \atop p\neq 2}2. \]  
If $\Lambda$ does not belong to the same $G$-genus as $\cO$, 
then $\prod_p c_p(g,R_p,\Lambda_p)=0$. 
\end{proposition} 
\begin{proof} 
We can prove (1) and the latter part of (2 ii) easily 
by \cite[Proposition 2.9]{HI83}. 
For any element $g$ of $G[f_{11}]$,  
$g$ is $G_p$-conjugate to 
\[ \left\{ \begin{array}{lcl} 
\gamma_p & \cdots & \mbox{ if } p\equiv 1,3\mbox{ or }5\mbox{ mod }8 \\ 
\gamma_p \mbox{ or } \delta_p & \cdots & \mbox{ if } p=2\mbox{ or }p\equiv 7\mbox{ mod }8 
\end{array} \right. \] 
for some elements $\gamma_p$ and $\delta_p \in G_p$. 
It follows from \cite[Proposition 20]{HI80} and \cite[Proposition 2.9]{HI83} that 
if $\Lambda$ is a $\bZ$-order of $Z(g)$ belonging to the same $G$-genus as $\cO$, 
then $c_p$ is as in the following table 
for each prime number $p$ satisfying the first column:  
\ \\ 
(i)\ If $p\mid D_1$, then we have 
\begin{center} 
\begin{tabular}{c|c|c} 
 & $c_p(\gamma_p,R_p,\cO_p)$&$c_p(\delta_p,R_p,\cO_p)$ \\ \hline 
$p\equiv 3\mbox{ mod }8$&$2$&$\times$  \\ 
$p\equiv 5\mbox{ mod }8$&$2$&$\times$  \\ 
$p\equiv 7\mbox{ mod }8$&$2$&$2$  \\ 
$p=2$&$1$&$1$  
\end{tabular} 
\end{center} 
(ii)\ If $p\mid D_2$, then we have 
\begin{center} 
\begin{tabular}{c|c|c} 
 & $c_p(\gamma_p,R_p,\cO_p)$&$c_p(\delta_p,R_p,\cO_p)$ \\ \hline 
$p\equiv 3\mbox{ mod }8$&$2$&$\times$  \\ 
$p\equiv 5\mbox{ mod }8$&$2$&$\times$  \\ 
$p\equiv 7\mbox{ mod }8$&$0$&$0$  \\ 
$p=2$&$1$&$1$  
\end{tabular} 
\end{center} 
(iii)\ If $p\nmid  D$, then we have 
\begin{center} 
\begin{tabular}{c|c|c} 
 & $c_p(\gamma_p,R_p,\cO_p)$&$c_p(\delta_p,R_p,\cO_p)$ \\ \hline 
$p\equiv 1\mbox{ mod }8$&$1$&$\times$  \\  
$p\equiv 3\mbox{ mod }8$&$1$&$\times$  \\ 
$p\equiv 5\mbox{ mod }8$&$1$&$\times$  \\ 
$p\equiv 7\mbox{ mod }8$&$1$&$0$  \\ 
$p=2$&$1$&$1$  
\end{tabular} 
\end{center} 
Also, $g$ is $Sp(2;\bR)$-conjugate to 
$g_1:=\alpha (\pi /4,3\pi /4)$, ${g_1}^{-1}=\alpha (-\pi /4,-3\pi /4)$, 
$g_2:=\alpha (\pi /4,-3\pi /4)$, or ${g_2}^{-1}=\alpha (-\pi /4,3\pi /4)$. 
Since $g^2$ belongs to $G[f_6]$, 
$g$ is $Sp(2;\bR)$-conjugate to $g_1$ or $g_1^{-1}$ if $Z_0(g^2)$ is indefinite, 
and $g$ is $Sp(2;\bR)$-conjugate to $g_2$ or $g_2^{-1}$ if $Z_0(g^2)$ is definite. 
We take all combinations of $G_p$-conjugacy classes for all $p$, and 
also take $Sp(2;\bR)$-conjugation out of ^^ ^^ $g_1$ or $g_1^{-1}$" or 
^^ ^^ $g_2$ or $g_2^{-1}$", 
according as $Z_0(g^2)$ is indefinite or definte. 
\end{proof} 
We see from Proposition \ref{prop:H11} that $H_{11}=0$ if 
$D_2(7;8)\neq\emptyset$. 
Hereafter, we assume $D_2(7;8)=\emptyset$. 
We obtain from Proposition \ref{prop:H11} that 
\[ H_{11}=c_{k,j}\cdot \sum _{\{ g\}_G} J_0'(g)\cdot M_G(\cO )\cdot \prod_p c_p(g,R_p,\cO _p), \] 
\[ \sum _{\{ g\}_G} J_0'(g)=
\Big( J_0'(g_1)+J_0'(g_1^{-1})+J_0'(g_2)+J_0'(g_2^{-1})\Big) 
\cdot 2^{\sharp D_1(7;8)}. \] 
We have $C_0(g;Sp(2;\bR))=\{ 1_4\}$ for any $g$, and \\ 
\scalebox{0.9}[1]{
$J_0'(g_1)+J_0'({g_1}^{-1})+J_0'(g_2)+J_0'({g_2}^{-1})=
\left\{ \begin{array}{ccl} 
{[1,0,0,-1;4]_k}&\cdots & \mbox{if }j\equiv 0 \mbox{ mod } 8 \\ 
{[-1,1,0,0;4]_k}&\cdots & \mbox{if }j\equiv 2 \mbox{ mod } 8 \\ 
{[-1,0,0,1;4]_k}&\cdots & \mbox{if }j\equiv 4 \mbox{ mod } 8 \\ 
{[1,-1,0,0;4]_k}&\cdots & \mbox{if }j\equiv 6 \mbox{ mod } 8 
\end{array} \right.$} . \\   
(cf. (b-1) in \cite{Wak}). 
We have 
\[ M_G(\cO )=\frac{1}{8}, \hspace{10pt} \text{(cf. (3.25) in \cite{Has84})} \] 
\[ \prod_p c_p(g,R_p,\cO _p)=\prod _{p \mid D \atop p\neq 2}2 \]  
for any $g$. 
Hence we can obtain $H_{11}$ as in Theorem \ref{thm:main}.

%%%%%%%%%%%%%%%%%%%%%%%%%%%%%%%%%%%%%%%%%%%%%%%%%%%%%%%%%%%%%%%%%%%%%%%%%%
%
\subsection{The contribution $\mathbf{H_{12}}$} \label{subsec:H12} 
%
%%%%%%%%%%%%%%%%%%%%%%%%%%%%%%%%%%%%%%%%%%%%%%%%%%%%%%%%%%%%%%%%%%%%%%%%%%

In this subsection, we evaluate the contribution of $G[f_{12}]$, 
where $f_{12}(x)=x^4-x^2+1$. 
Note that $G[f_{12}]=\emptyset$ if and only if $D(1;12)\not=\emptyset$. 
Hereafter, we assume $D(1;12)=\emptyset$. 
We see from \cite[Proposition 21]{HI80} and \cite[Proposition 2.10]{HI83} that 
if $D_2(11;12)\not=\emptyset$, then no $G$-conjugacy classes of $G[f_{12}]$ 
appear in the first sum of Theorem \ref{thm:elliptic}. 
Hereafter, we assume that $D_2(11;12)=\emptyset$. 

The set $G[f_{12}]$ consists of four $Sp(2;\bR)$-conjugacy classes represented by 
$h:=\alpha (\pi /6,5\pi /6)$, $h^{-1}=\alpha (-\pi /6,-5\pi /6)$, 
$h':=\alpha (\pi /6,-5\pi /6)$, $h'^{-1}=\alpha (-\pi /6,5\pi /6)$. 
We have $C_0(g;Sp(2;\bR ))=\{ 1_4\}$ for any $g\in G[f_{12}]$ and 
\begin{align*} 
J_0'(h)+J_0'(h^{-1}) &= (-1)^{j/2+k}\cdot [1,-1,0;3]_j,  \\ 
J_0'(h')+J_0'(h'^{-1}) &= (-1)^{j/2}\cdot [0,-1,1;3]_{j+2k}. 
\end{align*} 
(cf. (b-1) in \cite{Wak}). 
Only one $G$-genus represented by $\cO$, 
where $\cO$ is the ring of integers of $Z(g)\simeq \bQ (\zeta_{12})$, 
appears in the second sum of Theorem \ref{thm:elliptic} 
and $M_G(\cO)=\frac{1}{12}$ \ (cf. (3.27) in \cite{Has84}). 
If $g$ is an element of $G[f_{12}]$, then $g^2$ belongs to $G[f_7]$. 
We can obtain the following proposition from \cite[Proposition 21]{HI80} and \cite[Proposition 2.10]{HI83}. 

\begin{proposition} 
\ \\ 
$(I)$\ the case where $D_1(11;12)=\emptyset$ and $\sharp D(5;12)$ is even $($resp. odd$)$ \\ 
$(i)$\ the case where $2\nmid D$ and $3\nmid D_1:$ \\ 
Two $G$-conjugacy classes $\{ g\}_G$ and $\{ g^{-1}\}_G$ appear 
in the first sum of Theorem \ref{thm:elliptic}. 
They are $Sp(2;\bR)$-conjugate to $h'$ and $h'^{-1}$ $($resp. $h$ and $h^{-1})$ 
respectively, and 
\[ \prod _p c_p(g,R_p,\cO _p)=\prod _p c_p(g^{-1},R_p,\cO _p)=
\prod _{p\in D(5;12)}2 \cdot \prod _{p\in D(7;12)}2. \]  
$(ii)$\ the case where $2\mid D_2$ and $3\nmid D_1:$ \\ 
Two $G$-conjugacy classes $\{ g\}_G$ and $\{ g^{-1}\}_G$ appear 
in the first sum of Theorem \ref{thm:elliptic}. 
They are $Sp(2;\bR)$-conjugate to $h$ and $h^{-1}$ $($resp. $h'$ and $h'^{-1})$ 
respectively, and 
\[ \prod _p c_p(g,R_p,\cO _p)=\prod _p c_p(g^{-1},R_p,\cO _p)=
\prod _{p\in D(5;12)}2 \cdot \prod _{p\in D(7;12)}2. \]  
$(iii)$\ the case where $2\nmid D_1$ and $3\mid D_1:$ \\ 
Four $G$-conjugacy classes appear in the first sum of Theorem \ref{thm:elliptic}. 
They are $Sp(2;\bR)$-conjugate to $h,h^{-1},h',h'^{-1}$ respectively. \\ 
If $g$ is $Sp(2;\bR)$-conjugate to $h$ $($resp. $h')$, then 
\[ \prod _p c_p(g,R_p,\cO _p)=\prod _p c_p(g^{-1},R_p,\cO _p)=
\prod _{p\in D(5;12)}2 \cdot \prod _{p\in D(7;12)}2. \]  
If $g$ is $Sp(2;\bR)$-conjugate to $h'$ $($resp. $h)$, then 
\[ \prod _p c_p(g,R_p,\cO _p)=\prod _p c_p(g^{-1},R_p,\cO _p)=
2\cdot \prod _{p\in D(5;12)}2 \cdot \prod _{p\in D(7;12)}2. \]  
$(iv)$\ the case where $2\nmid D_1$ and $3\mid D_1:$ \\ 
Four $G$-conjugacy classes appear in the first sum of Theorem \ref{thm:elliptic}. 
They are $Sp(2;\bR)$-conjugate to $h,h^{-1},h',h'^{-1}$ respectively. \\ 
If $g$ is $Sp(2;\bR)$-conjugate to $h$ $($resp. $h')$, then 
\[ \prod _p c_p(g,R_p,\cO _p)=\prod _p c_p(g^{-1},R_p,\cO _p)=
2\cdot \prod _{p\in D(5;12)}2 \cdot \prod _{p\in D(7;12)}2. \]  
If $g$ is $Sp(2;\bR)$-conjugate to $h'$ $($resp. $h)$, then 
\[ \prod _p c_p(g,R_p,\cO _p)=\prod _p c_p(g^{-1},R_p,\cO _p)=
\prod _{p\in D(5;12)}2 \cdot \prod _{p\in D(7;12)}2. \]  
$(v)$\ the case where $2\mid D_1$ and $3\nmid D_1:$ \\ 
Four $G$-conjugacy classes appear in the first sum of Theorem \ref{thm:elliptic}. 
They are $Sp(2;\bR)$-conjugate to $h,h^{-1},h',h'^{-1}$ respectively. \\ 
If $g$ is $Sp(2;\bR)$-conjugate to $h$ $($resp. $h')$, then 
\[ \prod _p c_p(g,R_p,\cO _p)=\prod _p c_p(g^{-1},R_p,\cO _p)=
2\cdot \prod _{p\in D(5;12)}2 \cdot \prod _{p\in D(7;12)}2. \]  
If $g$ is $Sp(2;\bR)$-conjugate to $h'$ $($resp. $h)$, then 
\[ \prod _p c_p(g,R_p,\cO _p)=\prod _p c_p(g^{-1},R_p,\cO _p)=
\prod _{p\in D(5;12)}2 \cdot \prod _{p\in D(7;12)}2. \]  
$(vi)$\ the case where $2\mid D_1$ and $3\mid D_1:$ \\ 
Eight $G$-conjugacy classes appear in the first sum of Theorem \ref{thm:elliptic}. 
Each two of them are $Sp(2;\bR)$-conjugate to $h,h^{-1},h',h'^{-1}$ respectively. \\ 
If $g$ is $Sp(2;\bR)$-conjugate to $h$ $($resp. $h')$, then 
\[ \prod _p c_p(g,R_p,\cO _p)=\prod _p c_p(g^{-1},R_p,\cO _p)= 
\begin{cases} 
4\cdot \prod\limits _{p\in D(5;12)}2 \cdot \prod\limits _{p\in D(7;12)}2 \\ 
\prod\limits _{p\in D(5;12)}2 \cdot \prod\limits _{p\in D(7;12)}2. 
\end{cases} \]  
If $g$ is $Sp(2;\bR)$-conjugate to $h'$ $($resp. $h)$, then 
\[ \prod _p c_p(g,R_p,\cO _p)=\prod _p c_p(g^{-1},R_p,\cO _p)= 
\begin{cases} 
2\cdot \prod\limits _{p\in D(5;12)}2 \cdot \prod\limits _{p\in D(7;12)}2 \\ 
2\cdot \prod\limits _{p\in D(5;12)}2 \cdot \prod\limits _{p\in D(7;12)}2. 
\end{cases} \]  
$(II)$\ the case where $D_1(11;12)\neq \emptyset$ \\ 
$(i)$\ the case where $2\nmid D_1$ and $3\nmid D_1:$ \\ 
The number of $G$-conjugacy classes which appear in the first sum of Theorem \ref{thm:elliptic} is $2^{\sharp D_1(11;12) +1}$. 
They are $Sp(2;\bR)$-conjugate to $h,h^{-1},h',h'^{-1}$. 
All of them  satisfy 
\[ \prod\limits_p c_p(g,R_p,\cO_p)=\prod\limits_{p\in D(5;12)}2\cdot \prod\limits_{p\in D(7;12)}2\cdot \prod\limits_{p\in D_1(11;12)}2.  \] 
$(ii)$\ the case where ^^ ^^ $2\mid D_1$ and $3\nmid D_1$" of ^^ ^^ $2\nmid D_1$ and $3\mid D_1$" $:$ \\ 
The number of $G$-conjugacy classes which appear in the first sum of Theorem \ref{thm:elliptic} is $2^{\sharp D_1(11;12) +2}$. 
They are $Sp(2;\bR)$-conjugate to $h,h^{-1},h',h'^{-1}$. 
In each case,  \\ 
$2^{\sharp D_1(11;12) -1}$ $G$-conjugacy classes satisfy 
\[ \prod\limits_p c_p(g,R_p,\cO_p)=\prod\limits_{p\in D(5;12)}2\cdot \prod\limits_{p\in D(7;12)}2\cdot \prod\limits_{p\in D_1(11;12)}2\cdot 2, \] 
$2^{\sharp D_1(11;12) -1}$ $G$-conjugacy classes satisfy  
\[ \prod\limits_p c_p(g,R_p,\cO_p)=\prod\limits_{p\in D(5;12)}2\cdot \prod\limits_{p\in D(7;12)}2\cdot \prod\limits_{p\in D_1(11;12)}2.  \] 
$(iii)$\ the case where $2\mid D_1$ and $3\mid D_1:$ \\ 
The number of $G$-conjugacy classes which appear in the first sum of Theorem \ref{thm:elliptic} is $2^{\sharp D_1(11;12) +3}$. 
They are $Sp(2;\bR)$-conjugate to $h,h^{-1},h',h'^{-1}$. 
In each case,  \\ 
$2^{\sharp D_1(11;12) -1}$ $G$-conjugacy classes satisfy 
\[ \prod\limits_p c_p(g,R_p,\cO_p)=\prod\limits_{p\in D(5;12)}2\cdot \prod\limits_{p\in D(7;12)}2\cdot \prod\limits_{p\in D_1(11;12)}2\cdot 2^2, \] 
$2^{\sharp D_1(11;12)}$ $G$-conjugacy classes satisfy  
\[ \prod\limits_p c_p(g,R_p,\cO_p)=\prod\limits_{p\in D(5;12)}2\cdot \prod\limits_{p\in D(7;12)}2\cdot \prod\limits_{p\in D_1(11;12)}2\cdot 2, \] 
$2^{\sharp D_1(11;12) -1}$ $G$-conjugacy classes satisfy  
\[ \prod\limits_p c_p(g,R_p,\cO_p)=\prod\limits_{p\in D(5;12)}2\cdot \prod\limits_{p\in D(7;12)}2\cdot \prod\limits_{p\in D_1(11;12)}2. \] 
\end{proposition}

%%%%%%%%%%%%%%%%%%%%%%%%%%%%%%%%%%%%%%%%%%%%%%%%%%%%%%%%%%%%%%%%%
%%%%%%%%%%%%%%%%%%%%%%%%%%%%%%%%%%%%%%%%%%%%%%%%%%%%%%%%%%%%%%%%%
%
\section{The contribution of non-semi-simple conjugacy classes} 
\label{sec:nonsemisimple} 
%
%%%%%%%%%%%%%%%%%%%%%%%%%%%%%%%%%%%%%%%%%%%%%%%%%%%%%%%%%%%%%%%%%
%%%%%%%%%%%%%%%%%%%%%%%%%%%%%%%%%%%%%%%%%%%%%%%%%%%%%%%%%%%%%%%%%

In this section, we evaluate $I(\Gamma ^{(u)})_{k,j}$ and $I(\Gamma^{(qu)})_{k,j}$,  
i.e. the contributions of non-semi-simple conjugacy classes   
(cf. section \ref{seintro}).  
We prove $I_1$, $I_2$ and $I_3$ of Theorem \ref{thm:main}. 
Since the class number of $\fO$ is one, 
any maximal two-sided ideal $\fA$ can be written as 
$\fA =\fO\pi =\pi\fO$ for some $\pi\in\fO$. 
By taking conjugation by $\begin{pmatrix} \pi & 0 \\ 0 & \overline{\pi}^{-1} \end{pmatrix} \in G$, 
we may regard $\Gamma =G\cap \begin{pmatrix} \fO & \fA \\ \fA ^{-1} & \fO \end{pmatrix}$.   

Let $P$ be the unique parabolic subgroup of $G$, up to $G$-conjugation, i.e. 
\begin{align*} 
P=\left\{ \left. \begin{pmatrix} a&0\\0&\overline{a}^{-1} \end{pmatrix} 
\begin{pmatrix} 1&b\\0&1 \end{pmatrix} \right| a\in B^{\times}, b\in B^0 \right\}. 
\end{align*} 
We can prove that $\Gamma\backslash \fH_2$ has only one $0$-dimensional cusp 
in the same way as \cite[Proposition 2]{Ara81}. 
Arakawa proved Lemma \ref{lem:cusp} below in his master thesis \cite[Proposition 7]{Ara75}. 
\begin{lemma} \label{lem:cusp} 
We have $G=P\cdot \Gamma$. 
\end{lemma} 
\begin{proof} 
We take any $\gamma=\begin{pmatrix} a&b\\c&d \end{pmatrix} \in G$. 
There are some $\gamma$, $\delta \in \fO$ 
such that $c^{-1}d=\pi ^{-1}\gamma^{-1}\delta$. 
We can assume that there are some $u,v\in\fO$ such that $\gamma u+\delta v=1$. 
If we put 
\begin{align*} 
\tau :=\begin{pmatrix} 
\overline{v}-\overline{\pi^{-1}u}v\gamma\pi & \overline{\pi^{-1}u}(1-v\delta ) \\ 
\gamma\pi & \delta 
\end{pmatrix}, 
\end{align*} 
then we have $\tau\in\Gamma$ and $\sigma\tau^{-1}\in P$. 
\end{proof} 
By using Lemma \ref{lem:cusp}, 
we can prove Proposition \ref{prop:normalform}  below 
in the same way that Hashimoto proved it when $D_2=1$ in \cite[Lemma 1.2]{Has84}. 
\begin{proposition} \label{prop:normalform} 
If $\gamma$ is an element of $\Gamma^{(u)}\sqcup \Gamma^{(qu)}$, 
then $\gamma$ is $\Gamma$-conjugate to an element of 
the form: 
\begin{align*} 
\gamma (a,b) =\begin{pmatrix} a&0\\0&a \end{pmatrix} 
\begin{pmatrix} 1&b\\0&1 \end{pmatrix}, 
\end{align*} 
where $a\in\fO^{\times}$ is a root of unity and $b\in\fA^0-\{ 0\}$.  
\end{proposition}  

If $\gamma\in\Gamma^{(u)}$, then $a=\pm 1$ and the principal polynomial of $\gamma$ 
is $f_1(x)=(x-1)^4$ or $f_1(-x)$. 
We put $I_1=I(\Gamma^{(u)})_{k,j}$. 
If $\gamma\in\Gamma^{(qu)}$, then $a$ is a primitive 4-th, 3-rd, or 6-th root of unity  
and the principal polynomial of $\gamma$ is 
$f_6(x)=(x^2+1)^2$, $f_7(x)=(x^2+x+1)^2$ or $f_7(-x)$ respectively. 
We denote by $I_2$ (resp. $I_3$) the contribution of elements of 
$\Gamma^{(qu)}$ whose principal polynomial is $f_6(x)$ (resp. $f_7(\pm x)$).   
We evaluate $I_1$ in subsection \ref{subsec:I1}, and 
$I_2$ and $I_3$ in subsection \ref{subsec:I2} and \ref{subsec:I3}. 
We use the notation 
\[ \gamma (\theta , t):= \begin{pmatrix} \cos\theta & \sin\theta & 0&0 \\ -\sin\theta & \cos\theta &0&0 \\ 0&0&\cos\theta & \sin\theta \\ 0&0&-\sin\theta &\cos\theta \end{pmatrix} \begin{pmatrix} 1&0&t&0 \\ 0&1&0&t \\ 0&0&1&0 \\ 0&0&0&1 \end{pmatrix} . \] 

We summarize some lemmas which are used in subsection \ref{subsec:I2}. 
These lemmas were proved in the case of $D_2=1$ by Hashimoto \cite{Has84}. 
The following lemmas are  easy generalizations of them and can be proved 
in the almost same method, so we omit the proof. 
Let $a$ be a primitive 3-rd or 4-th or 6-th root of unity. 
We put $F:=\bQ (a)$ and denote by $\mathcal{O}_F$ the ring of integers of $F$ 
and by $d$ the discriminant of $F$. 
$F$ is isomorphic to $\bQ (\sqrt{-1})$ or $\bQ (\sqrt{-3})$. 
\begin{lemma} 
Let $\gamma $ be an element of $\Gamma $ of the form 
$\gamma (a,b)$ in Proposition \ref{prop:normalform}. 
Then we have \\ 
(1)\ If $\beta $ is an element of $B^{\times }$ such that 
$\beta a=\overline{a}\beta$, then we have $B=F\oplus F\beta$. \\ 
(2)\ If we express $b\in \mathfrak{A}^0$ as 
$b=x\sqrt{d}+y\beta$\ $(x\in \bQ , y\in F)$, 
then the Jordan decomposition $\gamma (a,b)=\gamma _s\cdot \gamma _u$ 
is given by 
\[ \gamma _s=\begin{pmatrix} a&0 \\ 0&a \end{pmatrix} 
\begin{pmatrix} 1&y\beta \\ 0&1 \end{pmatrix} ,\quad 
\gamma _u=\begin{pmatrix} 1&x\sqrt{d} \\ 0&1 \end{pmatrix} . \] 
\end{lemma} 

\begin{lemma} \label{lem:Eichler} 
If we put, for a fixed $a$ as above, 
\[ C(a):= \{ x^{-1}a x | x\in B^{\times} \} , \] 
then we have 
\[ \sharp ((C(a)\cap \fO)/\sim _{\fO^{\times}})  
=\prod _{p\mid D} (1-\left( \frac{F}{p}\right) ). \] 
\end{lemma} 

\begin{lemma}   
Let $\gamma _i=\gamma (a_i,b_i)$\ $(i=1,2)$ be two elements of $\Gamma^{qu}$ 
of the form of Proposition \ref{prop:normalform}. 
If $\gamma _1$ and $\gamma _2$ are $\Gamma$-conjugate, 
then $a_1$ and $a_2$ are $\fO^{\times}$-conjugate. 
\end{lemma} 

\begin{lemma}  \label{lem:La}  
Let $\gamma _i=\gamma (a,b_i)$\ $(i=1,2)$ be two elements of $\Gamma^{(qu)}$ 
of the form of Proposition \ref{prop:normalform}. 
We put 
\[ L_{\fA}(a):=\{ a^{-1}za-z | z\in \fA ^0 \} . \] 
Then $\gamma (a,b_1)$ and $\gamma (a,b_2)$ are $\Gamma$-conjugate 
if and only if $b_1-b_2\in L_{\fA}(a)$. 
\end{lemma}

\subsection{The contribution $\mathbf{I_1}$}  \label{subsec:I1} 
%
%%%%%%%%%%%%%%%%%%%%%%%%%%%%%%%%%%%%%%%%%%%%%%%%%%%%%%%%%%%%%%%%%%%%%%%%%%%

In this subsection, we evaluate the contribution $I_1$. 
We define the following four subsets of $\Gamma$:  
\begin{align*} 
F_1 &:= \left\{ \left. \begin{pmatrix} 1_2&S \\ 0_2&1_2 \end{pmatrix} 
\right| S\in SM_2(\bR ), \mathrm{det}S\neq 0, S:\mathrm{definite} 
\right\} \cap \Gamma , \\ 
F_2 &:= \left\{ \left. \begin{pmatrix} 1_2&S \\ 0_2&1_2 \end{pmatrix} 
\right| S\in SM_2(\bR ), \mathrm{det}S\neq 0, S:\mathrm{indefinite}, 
-\mathrm{det}S \not\in (\bQ ^{\times })^2  
\right\} \cap \Gamma , \\ 
F_3 &:= \left\{ \left. \begin{pmatrix} 1_2&S \\ 0_2&1_2 \end{pmatrix} 
\right| S\in SM_2(\bR ), \mathrm{det}S\neq 0, S:\mathrm{indefinite}, 
-\mathrm{det}S \in (\bQ ^{\times })^2  
\right\} \cap \Gamma , \\ 
F_4 &:= \left\{ \left. \begin{pmatrix} 1_2&S \\ 0_2&1_2 \end{pmatrix} 
\right| S\in SM_2(\bR ), \mathrm{det}S=0
\right\} \cap \Gamma ,  
\end{align*} 
where we denote by $SM_2(\bR)$ the set of all symmetric matrices of degree 2 
over $\bR$. 
We denote by $C^{(u)}$ the set of all $\Gamma$-conjugacy classes of $\Gamma^{(u)}$. 
We can prove the following proposition by Proposition \ref{prop:normalform}. 
\begin{proposition} \label{Proposition:Unipotent1} 
We can decompose $C^{(u)}$ as  
\begin{align*} 
C^{(u)} = \bigsqcup _{i=1}^4 \Big( \bigsqcup _{\gamma \in F_i/\sim _{\Gamma }} 
\Big\{ \ \{ \gamma \} _{\Gamma } \Big\} \Big) , 
\end{align*} 
where $F_i/\sim _{\Gamma }$ denotes a complete system of representatives of 
$\Gamma$-conjugacy calsses of $F_i$ and 
$\{ \gamma \} _{\Gamma}$ denotes the $\Gamma$-conjugacy class represented by $\gamma$. 
\end{proposition} 

\begin{proof} 
Take an arbitrary $\{ \gamma '\} _{\Gamma } \in C_u$. 
By Proposition \ref{prop:normalform}, we have some $x\in\Gamma$ 
such that 
$x^{-1}\gamma 'x=\pm \begin{pmatrix} 1&b\\0&1 \end{pmatrix}$, $b\in\fA^0-\{ 0\}$. 
Identifying $x^{-1}\gamma 'x$ and its image by $\phi$ in $Sp_2(\bR)$, 
it is contained in some $F_i$, 
so we have 
\[ \{ \gamma '\} _{\Gamma } =\{ x^{-1}\gamma 'x\} _{\Gamma} 
\in \bigsqcup _{\gamma \in F_i/\sim _{\Gamma}} \Big\{ \ \{ \gamma \} _{\Gamma} \Big\} . \] 
\end{proof}  

However, especially in the case of our $\Gamma$, we have $F_3=F_4=\emptyset$ and 
\[ I_1= c_{k,j} \cdot 
\sum _{i=1}^2 \mathrm{vol}(C_0(\gamma_i ;\Gamma )\backslash C_0(\gamma_i ;Sp(2;\bR )))%$ \\ 
%\hspace{65mm} 
\cdot \lim _{s\rightarrow +0} \sum _{\gamma '\in F_i/\sim} 
\frac{J_0(\gamma ';s)}{[C(\gamma ';\Gamma ):\pm C_0(\gamma ';\Gamma )]}, \] 
where $\gamma_i$ is an any element of $F_i$. (cf. \cite[Theorem 3.1]{Wak}). 
By using the formula of \cite[(e-2),(e-3)]{Wak}, we have 
\begin{align*} 
\lim _{s\rightarrow +0} \sum _{\gamma '\in F_i/\sim} 
\frac{J_0(\gamma ';s)}{[C(\gamma ';\Gamma ):\pm C_0(\gamma ';\Gamma )]}
=\left\{ \begin{array}{ll} 
c_{k,j}^{-1} \cdot \frac{j+1}{2^2\pi } \cdot 
          \frac{1}{[\tilde{\Gamma}:\tilde{\Gamma}_+]} \cdot 
          \frac{\mathrm{vol}(\tilde{\Gamma}_+\backslash \fH _1)}{\mathrm{vol}(L)}
& \cdots i=1 \\ 
0&\cdots i=2 
\end{array} \right. 
\end{align*} 
Here, we define the notations as follows. 

We define a lattice $L$ in $SM_2(\bR)$ by 
\[ \left\{ \left. \begin{pmatrix} 1_2&X\\0_2&1_2 \end{pmatrix} \right| X\in L\right\} 
=\left\{ \left. \begin{pmatrix} 1_2&S\\0_2&1_2 \end{pmatrix} \right| S\in SM_2(\bR ) \right\} \cap \Gamma . \] 
We put 
\begin{align*} 
C_0(\gamma _1;Sp(2;\bR )) &= \left\{ \left. \begin{pmatrix} 1_2&S\\0_2&1_2 \end{pmatrix} \right| S\in SM_2(\bR ) \right\} , & 
C_0(\gamma_1;\Gamma ) &= C_0(\gamma _1;Sp(2;\bR )) \cap \Gamma \\ 
 & & &=\left\{ \left. \begin{pmatrix} 1_2&X\\0_2&1_2 \end{pmatrix} \right| X\in L\right\} 
\end{align*} 
and 
\[ \mathrm{vol}(L) :=\mathrm{vol}(C_0(\gamma_1;\Gamma )\backslash C_0(\gamma_1;G(\bR ))) = \int _{L\backslash SM_2(\bR )} dx_{11}dx_{12}dx_{22} \] 
for $\begin{pmatrix} x_{11}&x_{12}\\x_{12}&x_{22} \end{pmatrix} \in SM_2(\bR)$.  
We put 
\begin{align*} 
\tilde{\Gamma} &= \left\{ \left. \begin{pmatrix} x&0\\0&\overline{x}^{-1} \end{pmatrix} \right| x\in B^{\times} \right\} \cap \Gamma , & 
\tilde{\Gamma}_+ &= \left\{ \left. \begin{pmatrix} x&0\\0&\overline{x}^{-1} \end{pmatrix} \in \tilde{\Gamma} \right| x\overline{x}>0 \right\} . 
\end{align*} 
We can identify $\tilde{\Gamma}_+$ as the subgroup of $GL_+(2;\bR )=\{ g\in GL(2;\bR ) | \mathrm{det}(g) >0\}$ and we define 
\[ \mathrm{vol}(\tilde{\Gamma}_+\backslash \fH_1) =\int _{\tilde{\Gamma}_+\backslash \fH_1} y^{-2}dxdy \] 
for $x+iy \in \fH_1$, where $\fH_1$ is the upper half plane $\{ z\in \bC | \mathrm{Im}(z)>0 \}$. 

It follows that we have 
\[ I_1= \frac{j+1}{2^3\pi} \cdot \frac{\mathrm{vol}(\tilde{\Gamma}_+\backslash \fH_1)}{[\tilde{\Gamma}:\tilde{\Gamma}_+]} . \] 
Noting that $\tilde{\Gamma}$ and $\tilde{\Gamma}_+$ are independent on a choice 
of pairs ($D_1$, $D_2$) for a fixed $D$, 
we see that the value $I_1$ is also independent on it. 
Hence we have 
\[ I_1=2^{-3}3^{-1}(j+1)\prod _{p\mid D} (p-1), \] 
which is the same value as in \cite[Theorem 6.1]{Wak}.

%%%%%%%%%%%%%%%%%%%%%%%%%%%%%%%%%%%%%%%%%%%%%%%%%%%%%%%%%%%%%%%%%
%
\subsection{The contribution $\mathbf{I_2}$} 
\label{subsec:I2} 
%
%%%%%%%%%%%%%%%%%%%%%%%%%%%%%%%%%%%%%%%%%%%%%%%%%%%%%%%%%%%%%%%%%

In this section, we evaluate the contribution $I_2$. 
Let $\gamma$ be an element of $\Gamma^{(qu)}$ whose principal polynomial is 
$f_6(x)=(x^2+1)^2$. 
Then $\gamma$ is $Sp(2;\bR)$-conjugate to an element of the form 
\begin{align*} 
\gamma (\pi /2,s)= 
\begin{pmatrix} 0 & 1 & 0 & 0 \\ 
-1 & 0 & 0 & 0 \\ 
0 & 0 & 0 & 1 \\ 
0 & 0 & -1 & 0 \end{pmatrix} 
\begin{pmatrix} 1&0&s&0 \\ 0&1&0&s \\ 0&0&1&0 \\ 0&0&0&1 \end{pmatrix} 
\end{align*} 
(cf. Proposition \ref{prop:normalform}) and corresponds to (f-3) of \cite{Wak}. 

We denote by  $C_6$ the set of all $\Gamma$-conjugacy classes of $\Gamma^{(qu)}$ 
whose principal polynomial is $f_6(x)$. 
Then we have the following  proposition: 
\begin{proposition} \label{Prop:ClassifyQU} 
We can decompose $C_6$ into disjoint union of $4N$ subsets as 
\begin{align*} 
C_6=\bigsqcup_{i=1}^N \bigsqcup_{j=1}^4 \Big( \bigsqcup_{\gamma\in F_{i,j}} 
\Big\{ \{ \gamma \} _{\Gamma} \Big\}\Big),  
\end{align*} 
where $N:=\prod\limits _{p\mid D} (1-\left( \frac{-1}{p}\right))$ 
and $F_{i,j}$ is defined as follows. 
\end{proposition} 
Let $a_1,\ldots a_N$ be a complete system of $\fO ^{\times}$-conjugacy classes 
of elemetns of $\fO$ of order $4$. 
(cf. Lemma \ref{lem:Eichler}). 
There exist some $x_i\in\bQ_{>0}$ and $\beta_i\in \fO^{0}$ depending on each $a_i$ 
such that  $F_{i,j}$ 's are given as one of the following four cases.  
Here we put 
\begin{align*} 
\delta (a_i,\gamma _1,\gamma _2) :=\begin{pmatrix} a_i&0\\0&a_i \end{pmatrix} 
\begin{pmatrix} 1&\gamma _1\beta _i\\0&1 \end{pmatrix} \cdot 
\begin{pmatrix} 1&\gamma _2x_ia_i\\0&1 \end{pmatrix}, 
\end{align*} 
where the symbol ^^ ^^ $\cdot$ " means the Jordan decomposition. 

\vspace{3mm} \noindent 
Case 1: 
\begin{align*} 
F_{i,1} &= \big\{ \delta (a_i,0,l)\ |\ l\in \bZ -\{ 0\} \big\} , & 
F_{i,2} &= \big\{ \delta (a_i,1,l)\ |\ l\in \bZ -\{ 0\} \big\} \\ 
F_{i,3} &= \big\{ \delta (a_i,a_i,l)\ |\ l\in \bZ -\{ 0\} \big\} , & 
F_{i,4} &= \big\{ \delta (a_i,1+a_i,l)\ |\ l\in \bZ -\{ 0\} \big\} 
\end{align*} 
All elements of $F_{i,j}$ are conjugate to $\gamma (\frac{\pi}{2},l)$ 
in $Sp (2,\bR)$. 

\vspace{3mm} \noindent 
Case 2: 
\begin{align*} 
F_{i,1} &= \big\{ \delta (a_i,0,2l)\ |\ l\in \bZ -\{ 0\} \big\} , & 
F_{i,2} &= \big\{ \delta (a_i,1,2l)\ |\ l\in \bZ -\{ 0\} \big\} \\ 
F_{i,3} &= \big\{ \delta (a_i,\tfrac{1}{2}a_i,2l+1)\ |\ l\in \bZ \big\} ,& 
F_{i,4} &= \big\{ \delta (a_i,1+\tfrac{1}{2}a_i,2l+1)\ |\ l\in \bZ \big\} 
\end{align*} 
All elements of $F_{i,1}$ and $F_{i,2}$ are conjugate to $\gamma (\frac{\pi}{2},l)$ 
in $Sp(2,\bR)$. 
All elements of $F_{i,3}$ and $F_{i,4}$ are conjugate to 
$\gamma (\frac{\pi}{2},l+\frac{1}{2})$ in $Sp(2,\bR)$. 

\vspace{3mm} \noindent 
Case 3: 
\begin{align*} 
F_{i,1} &= \big\{ \delta (a_i,0,2l)\ |\ l\in \bZ -\{ 0\} \big\} , & 
F_{i,2} &= \big\{ \delta (a_i,a_i,2l)\ |\ l\in \bZ -\{ 0\} \big\} \\ 
F_{i,3} &= \big\{ \delta (a_i,\tfrac{1}{2},2l+1)\ |\ l\in \bZ \big\} ,& 
F_{i,4} &= \big\{ \delta (a_i,\tfrac{1}{2}+a_i,2l+1)\ |\ l\in \bZ \big\} 
\end{align*} 
All elements of $F_{i,1}$ and $F_{i,2}$ are conjugate to $\gamma (\frac{\pi}{2},l)$ 
in $Sp(2,\bR)$. 
All elements of $F_{i,3}$ and $F_{i,4}$ are conjugate to 
$\gamma (\frac{\pi}{2},l+\frac{1}{2})$ in $Sp(2,\bR)$. 

\vspace{3mm} \noindent 
Case 4: 
\begin{align*} 
F_{i,1} &= \big\{ \delta (a_i,0,2l)\ |\ l\in \bZ -\{ 0\} \big\} , & 
F_{i,2} &= \big\{ \delta (a_i,1,2l)\ |\ l\in \bZ -\{ 0\} \big\} \\ 
F_{i,3} &= \big\{ \delta (a_i,\tfrac{1}{2}+\tfrac{1}{2}a_i,2l+1)\ |\ l\in \bZ \big\} , & 
F_{i,4} &= \big\{ \delta (a_i,\tfrac{1}{2}+\tfrac{3}{2}a_i,2l+1)\ |\ l\in \bZ \big\} 
\end{align*} 
All elements of $F_{i,1}$ and $F_{i,2}$ are conjugate to $\gamma (\frac{\pi}{2},l)$ 
in $Sp(2,\bR)$. 
All elements of $F_{i,3}$ and $F_{i,4}$ are conjugate to 
$\gamma (\frac{\pi}{2},l+\frac{1}{2})$ in $Sp(2,\bR)$.

\begin{proof} 
We take an arbitray $\{\gamma\}_{\Gamma} \in C_6$. 
By Proposition \ref{prop:normalform}, we have 
\begin{align*} 
\gamma \sim _{\Gamma} \begin{pmatrix} a&0\\0&a \end{pmatrix} 
\begin{pmatrix} 1&b\\0&1 \end{pmatrix} 
\end{align*} 
for some $a\in \fO$ of order 4 and $b\in \fA ^{0}-\{ 0\}$.  
By taking $\Gamma$-conjugation, we may have $a=a_i$ 
for some $i\in \{ 1,\ldots ,N\}$. 
Hence we have 
\begin{align*} 
C_6=\bigsqcup _{i=1}^{N}X_i, \quad 
X_i=\left\{ \left. \left\{ \begin{pmatrix} a_i&0\\0&a_i \end{pmatrix} 
\begin{pmatrix} 1&b\\0&1 \end{pmatrix} \right\} _{\Gamma}  \right| 
b\in \fA ^{0} -\{ 0\} \right\} . 
\end{align*} 
For each $X_i$, we simply put $a=a_i$. 
By Lemma \ref{lem:La}, each $X_i$ can be decomposed as 
\begin{align*} 
X_i=\bigsqcup _{b\in \fA ^{0} /L_{\fA}(a)} 
\left\{ \left\{ \begin{pmatrix} a&0\\0&a \end{pmatrix} 
\begin{pmatrix} 1&b\\0&1 \end{pmatrix} \right\} _{\Gamma} \right\} . 
\end{align*} 
We can describe the structure of $\fA ^{0}/L_{\fA}(a)$ by the same way as 
Hashimoto \cite{Has84} as follows. 
From Proposition 2.5 of \cite{Has84}, we have 
\begin{align*} 
\fO ^0 =\left\{ \begin{array}{l} \bZ \cdot \frac{a+\beta}{2}+\cO_F\beta \cdots \mbox{ if }2\nmid  D, 
\\ 
\bZ \cdot a+\cO_F\beta \cdots \mbox{ if }2\nmid  D \end{array} \right. 
\end{align*} 
for some $\beta$. 
So we have $\fO ^0\cap F^{\perp}=\cO_F\beta$ and $\cA ^0 \cap F^{\perp}$ is a $\cO_F$-
submodule of $\fO ^0 \cap F^{\perp}$. 
Since $\cO_F$ is P.I.D. and $\fO ^0\cap F^{\perp}$ is a free $\cO_F$-module 
of rank $1$, $\fA ^0 \cap F^{\perp}$ is also a free $\cO_F$-module of renk $1$. 
So we can write $\fA ^0 \cap F^{\perp}=\cO_F \beta '$ with some $\beta'$. 
Since $\fA^0/(\fA^0\cap F^{\perp})$ is a torsion-free $\bZ$-module, 
$\fA^0\cap F^{\perp}$ is a direct summand of $\fA^0$, that is, 
there exists some sub$\bZ$-module $M$ of $\fA^0$ and we can write 
$\fA^0=M\oplus (\fA^0\cap F^{\perp})$. 
The $\bZ$-module $M$ is free of rank $1$. 
A basis of $M$ can be expressed as the form: 
$xa+y\beta '$\quad ($x\in \bQ -\{0\}$, $y\in F$) 
because we have $B^0=\bQ a+ F\beta '$ with $\beta '$ mentioned above. 
So we can take 
$\rho _1:=xa+y\beta '$, $\rho _2:=\beta '$ and $\rho _3:=a\beta '$ 
as a basis of $\fA^0$. 
From the relation 
$-2y\beta'=a^{-1}\rho_1a-\rho_1\in L_{\fA}(a) \subset \fA^0\cap F^{\perp} 
=\cO_F \beta '$, 
we have $2y\in \cO_F=\bZ +\bZ a$. 
We divide the situation into two cases according as $y\in \cO_F$ of $\not\in \cO_F$. 
\\ 
(i)\ The case of $y\in \cO_F$. 
We can write $\rho _1=xa+y_1\beta'+y_2a\beta'$ with some $y_1,y_2\in \bZ$. 
So by replacing $\rho_1$, $\rho_1=xa$, $\rho_2=\beta'$, $\rho_3=a\beta'$ 
forms a besis of $\fA^0$, that is 
\[ \fA^0=\bZ xa\oplus \bZ \beta' \oplus \bZ a\beta '. \]  
(ii)\ The case of $y\not\in \cO_F$. 
We have $y=y_1+y_2a$,\quad $2y_1,2y_2\in \bZ$ and 
$\fA^0=\bZ\cdot (xa+y_1\beta'+y_1a\beta')\oplus \bZ\beta'\oplus \bZ a\beta'$. 
So $\fA^0$ is one of the following three cases: 
\begin{align*} 
\mbox{Case (ii a)\qquad } \fA^0 &= \bZ\cdot (xa+\frac{1}{2}a\beta ') \oplus \bZ \beta' \oplus \bZ a\beta' \\ 
\mbox{Case (ii b)\qquad } \fA^0 &= \bZ\cdot (xa+\frac{1}{2}\beta ') \oplus \bZ \beta' \oplus \bZ a\beta' \\ 
\mbox{Case (ii c)\qquad } \fA^0 &= \bZ\cdot (xa+\frac{1}{2}\beta'+\frac{1}{2}a\beta ') \oplus \bZ \beta' \oplus \bZ a\beta'. 
\end{align*} 
In each case, the structure of $L_{\fA}(a)$ and $\fA^0 /L_{\fA}(a)$ are given as follows: 

\vspace{2mm} \noindent 
Case (i): $L_{\fA}(a)=\{ 2m\beta'+2na\beta' | m,n \in \bZ \}$, 
$\fA^0/L_{\fA}(a)=\{ lxa | l\in \bZ\} \sqcup \{ lxa+\beta' |l\in \bZ \} 
\sqcup \{ lxa+a\beta' |l\in \bZ \} \sqcup \{ lxa+\beta'+a\beta' |l\in \bZ \}$. 

\vspace{2mm} \noindent 
Case (ii a):  $L_{\fA}(a)=\{ 2m\beta'+na\beta' | m,n \in \bZ \}$, 
$\fA^0/L_{\fA}(a)=\{ lxa | l\in 2\bZ\} \sqcup \{ lxa+\beta' |l\in 2\bZ \} 
\sqcup \{ lxa+\frac{1}{2}a\beta' |l\in 2\bZ +1\} 
\sqcup \{ lxa+\beta'+\frac{1}{2}a\beta' |l\in 2\bZ +1\}$. 

\vspace{2mm} \noindent 
Case (ii b):  $L_{\fA}(a)=\{ m\beta'+2na\beta' | m,n \in \bZ \}$, 
$\fA^0/L_{\fA}(a)=\{ lxa | l\in 2\bZ\} \sqcup \{ lxa+a\beta' |l\in 2\bZ \} 
\sqcup \{ lxa+\frac{1}{2}\beta' |l\in 2\bZ +1\} 
\sqcup \{ lxa+\frac{1}{2}\beta'+a\beta' |l\in 2\bZ +1\}$. 

\vspace{2mm} \noindent 
Case (ii c):  $L_{\fA}(a)=\{ m\beta'+na\beta' | m,n \in \bZ \}$, 
$\fA^0/L_{\fA}(a)=\{ lxa | l\in 2\bZ\} \sqcup \{ lxa+\beta' |l\in 2\bZ \} 
\sqcup \{ lxa+\frac{1}{2}\beta'+\frac{1}{2}a\beta' |l\in 2\bZ +1\} 
\sqcup \{ lxa+\frac{1}{2}\beta'+\frac{3}{2}a\beta' |l\in 2\bZ +1\}$. 

\vspace{2mm} \noindent 
Thus we have completed the proof of Proposition \ref{Prop:ClassifyQU}. 
\end{proof} 

The sets $F_{i,l}$'s are called \textit{families} in \cite{Has83}, \cite{Has84}, 
\cite{Wak}, etc. 
For each $F_{i,l}$, there exist $g_{i,l}\in Sp(2;\bR)$ and 
$\lambda\in\bR$ with $0\leq\lambda_{i,l}< 1$, such that 
\[ F_{i,l}=g_{i,l}\left. \left\{ \begin{pmatrix} 0&1&0&0 \\ -1&0&0&0 \\ 0&0&0&1 \\ 0&0&-1&0 \end{pmatrix} \begin{pmatrix} 1&0&n+\lambda_{i,l}&0 \\ 0&1&0&n+\lambda_{i,l} \\ 0&0&1&0 \\ 0&0&0&1 \end{pmatrix} \right| \begin{array}{cc} n\in \bZ \\ n+\lambda_{i,l} \neq 0 \end{array} \right\} g_{i,l}^{-1}. \] 
We define \\ 
\resizebox{0.95\textwidth}{!}{$C(F_{i,l};Sp(2;\bR )) := g_{i,l}\left. \left\{ \begin{pmatrix} \cos\theta & \sin\theta &0&0 \\ -\sin\theta & \cos\theta &0&0 \\ 0&0&\cos\theta & \sin\theta \\ 0&0&-\sin\theta & \cos\theta \end{pmatrix}  \begin{pmatrix} 1&0&t&0 \\ 0&1&0&t \\ 0&0&1&0 \\ 0&0&0&1 \end{pmatrix} \right| \theta , t \in \bR \right\} g_{i,l}^{-1}$}, \\ 
$C_0(F_{i,l};Sp(2;\bR )) := g_{i,l}\left. \left\{ \begin{pmatrix} 1&0&t&0 \\ 0&1&0&t \\ 0&0&1&0 \\ 0&0&0&1 \end{pmatrix} \right| t \in \bR \right\} g_{i,l}^{-1}$, \\ 
$C(F_{i,l};\Gamma ) := C(F_{i,l};Sp(2;\bR ))\cup \Gamma$, \\ 
$C_0(F_{i,l};\Gamma ) := C_0(F_{i,l};Sp(2;\bR ))\cup \Gamma$. \\ 
Then, from (f-3) in \cite{Wak}, we have 
\begin{align*} 
I_2 &= \sum _{i=1}^{N} \sum _{l=1}^{4} \frac{1}{2}\cdot 
\frac{\mathrm{vol}(C_0(F_{i,l};\Gamma )\backslash C_0(F_{i,l};Sp(2,\bR )))}{[C(F_{i,l};\Gamma ):\pm C_0(F_{i,l};\Gamma )]} \\ 
 &\hspace{150pt} \cdot (-2^{-3}(-1)^{j/2})\cdot (1-\sqrt{-1}\cot ^{*}\pi\lambda_{i,l}), 
\end{align*} 
where we put 
\[ \cot ^{*}\pi\lambda := \left\{ \begin{array}{ccl} 0 & \cdots &\text{ if }\lambda =0, \\ \cot\pi\lambda & \cdots & \text{ if }0\le \lambda < 1. \end{array} \right. \] 
We can verify that 
\begin{align*} 
C(F_{i,l};\Gamma ) &= g_{i,l}\left. \left\{ \pm \begin{pmatrix} 1&0&t&0 \\ 0&1&0&t \\ 0&0&1&0 \\ 0&0&0&1 \end{pmatrix}, \pm \begin{pmatrix} 0&1&0&0 \\ -1&0&0&0 \\ 0&0&0&1 \\ 0&0&-1&0 \end{pmatrix} \begin{pmatrix} 1&0&t&0 \\ 0&1&0&t \\ 0&0&1&0 \\ 0&0&0&1 \end{pmatrix} \right| t\in\bZ \right\} g_{i,l}^{-1}, \\ 
C_0(F_{i,l};\Gamma ) &= g_{i,l}\left. \left\{ \begin{pmatrix} 1&0&t&0 \\ 0&1&0&t \\ 0&0&1&0 \\ 0&0&0&1 \end{pmatrix} \right| t\in\bZ \right\} g_{i,l}^{-1} 
\end{align*} 
and 
\[ \mathrm{vol}\big( C_0(F_{i,l};\Gamma )\backslash C_0(F_{i,l};Sp(2;\bR ))\big) =1, \] 
\[ [C(F_{i,l};\Gamma ):\pm C_0(F_{i,l};\Gamma )]=2. \] 
Hence we have $I_2=-4N\cdot 2^{-5}(-1)^{j/2}$.

%%%%%%%%%%%%%%%%%%%%%%%%%%%%%%%%%%%%%%%%%%%%%%%%%%%%%%%%%%%%%%%%%
%
\subsection{The contribution $\mathbf{I_3}$} 
\label{subsec:I3} 
%
%%%%%%%%%%%%%%%%%%%%%%%%%%%%%%%%%%%%%%%%%%%%%%%%%%%%%%%%%%%%%%%%%

In this section, we evaluate the contribution $I_3$. 
We consider the contribution of elements whose principal polynmoimals are 
$f_7(x)=(x^2+x+1)^2$ and double it to obtain $I_3$. 
Let $\gamma$ be an element of $\Gamma^{(qu)}$ whose principal polynomial is 
$f_7(x)$. 
Then $\gamma$ is $Sp(2;\bR)$-conjugate to an element of the form 
\begin{align*} 
\gamma (2\pi /3, s)= 
\begin{pmatrix} -1/2 & \sqrt{3}/2 & 0 & 0 \\ 
-\sqrt{3}/2 & -1/2 & 0 & 0 \\ 
0 & 0 & -1/2 & \sqrt{3}/2 \\ 
0 & 0 & -\sqrt{3}/2 & -1/2 \end{pmatrix} 
\begin{pmatrix} 1&0&s&0 \\ 0&1&0&s \\ 0&0&1&0 \\ 0&0&0&1 \end{pmatrix} 
\end{align*} 
(cf. Proposition \ref{prop:normalform}) and corresponds to (f-3) of \cite{Wak}. 

We denote by  $C_7$ the set of all $\Gamma$-conjugacy classes of $\Gamma^{(qu)}$ 
whose principal polynomial is $f_7(x)$. 
By the same way as Proposition \ref{Prop:ClassifyQU}, 
we can prove the following proposition: 

\begin{proposition} \label{prop:classifyI3} 
We can decompose $C_7$ into disjoint union of $3N$ subsets as 
\begin{align*} 
C_7=\bigsqcup_{i=1}^N \bigsqcup_{l=1}^3 \Big( \bigsqcup_{\gamma\in F_{i,l}} 
\Big\{ \{ \gamma \} _{\Gamma} \Big\}\Big),  
\end{align*} 
where $N:=\prod\limits _{p\mid D} (1-\left( \frac{-3}{p}\right))$ 
and $F_{i,j}$ is defined as follows. 
\end{proposition} 
Let $a_1,\ldots a_N$ be a complete system of $\fO ^{\times}$-conjugacy classes 
of elemetns of $\fO$ of order $3$. 
(cf. Lemma \ref{lem:Eichler}). 
There exist some $x_i\in\bQ_{>0}$ and $\beta_i\in \fO^{0}$ depending on each $a_i$ 
such that  $F_{i,j}$ 's are given as one of the following two cases.  
Here we put 
\begin{align*} 
\delta (a_i,\gamma _1,\gamma _2) :=\begin{pmatrix} a_i&0\\0&a_i \end{pmatrix} 
\begin{pmatrix} 1&\gamma _1\beta _i\\0&1 \end{pmatrix} \cdot 
\begin{pmatrix} 1&\gamma _2x_i\sqrt{-3}\\0&1 \end{pmatrix}, 
\end{align*} 
where the symbol ^^ ^^ $\cdot$ " means the Jordan decomposition. 

\vspace{3mm} \noindent 
Case 1: 
\begin{align*} 
F_{i,1} &= \big\{ \delta (a_i,0,n)\ |\ n\in \bZ -\{ 0\} \big\} , & 
F_{i,2} &= \big\{ \delta (a_i,1,n)\ |\ n\in \bZ -\{ 0\} \big\} \\ 
F_{i,3} &= \big\{ \delta (a_i,2,n)\ |\ n\in \bZ -\{ 0\} \big\} , & 
\end{align*} 
All elements of $F_{i,l}$'s are $Sp(2;\bR)$-conjugate to $\gamma(2\pi/3,n)$. 
 
\vspace{3mm} \noindent 
Case 2: 
\begin{align*} 
F_{i,1} &= \big\{ \delta (a_i,0,3n)\ |\ n\in \bZ -\{ 0\} \big\} , & 
F_{i,2} &= \big\{ \delta (a_i,(1+2a)/3,3n+1)\ |\ n\in \bZ \big\} \\ 
F_{i,3} &= \big\{ \delta (a_i,(2+a)/3,3n+2)\ |\ l\in \bZ \big\} ,& 
\end{align*} 
All elements of each $F_{i,l}$ are $Sp(2;\bR)$-conjugate to 
$\gamma (2\pi/3,n+(l-1)/3)$.  

\vspace{5mm} 
For each $F_{i,l}$, we define $g_{i,l}$, $\lambda_{i,l}$, $C(F_{i,l};Sp(2;\bR))$, 
$C_0(F_{i,l};Sp(2;\bR))$, $C(F_{i,l};\Gamma)$ and $C_0(F_{i,l};\Gamma)$ 
in the same way as in subsection \ref{subsec:I2}. 
Then, from (f-3) in \cite{Wak}, we have 
\begin{align*} 
I_3 &= \sum _{i=1}^{N} \sum _{l=1}^{3} \cdot 
\frac{\mathrm{vol}(C_0(F_{i,l};\Gamma )\backslash C_0(F_{i,l};Sp(2,\bR )))}{[C(F_{i,l};\Gamma ):\pm C_0(F_{i,l};\Gamma )]} \\ 
 &\hspace{140pt} \cdot (-2^{-1}3^{-1}[1,-1,0;3]_j)\cdot (1-\sqrt{-1}\cot ^{*}\pi\lambda_{i,l}).  
\end{align*} 
We can verify that 
\begin{align*} 
C(F_{i,l};\Gamma ) &= g_{i,l}\left. \Big\{ \pm \gamma (\theta ,t) \right| \theta =0, \pi/3, 2\pi/3, \ t\in\bZ \Big\} g_{i,l}^{-1}, \\ 
C_0(F_{i,l};\Gamma ) &= g_{i,l}\left. \Big\{ \gamma(0,t) \right| t\in\bZ \Big\} g_{i,l}^{-1} 
\end{align*} 
and 
\[ \mathrm{vol}\big( C_0(F_{i,l};\Gamma )\backslash C_0(F_{i,l};Sp(2;\bR ))\big) =1,\hspace{8pt} [C(F_{i,l};\Gamma ):\pm C_0(F_{i,l};\Gamma )]=3. \] 
Hence we have $I_3=-3N\cdot 2^{-1}3^{-2}[1,-1,0;3]_j$.

%%%%%%%%%%%%%%%%%%%%%%%%%%%%%%%%%%%%%%%%%%%%%%%%%%%%%%%%%%%%%%%%%%%%%%%%%%
%%%%%%%%%%%%%%%%%%%%%%%%%%%%%%%%%%%%%%%%%%%%%%%%%%%%%%%%%%%%%%%%%%%%%%%%%%
%
\section{Numerical examples} \label{sec:example} 
%
%%%%%%%%%%%%%%%%%%%%%%%%%%%%%%%%%%%%%%%%%%%%%%%%%%%%%%%%%%%%%%%%%%%%%%%%%%
%%%%%%%%%%%%%%%%%%%%%%%%%%%%%%%%%%%%%%%%%%%%%%%%%%%%%%%%%%%%%%%%%%%%%%%%%%

In this section, we give some numerical examples of 
$\dim _{\bC} S_{k,j}(\Gamma (D_1,D_2))$ for various $D_1,D_2$. 
The tables for $D=D_1=6,10,15$ appeared in \cite{Wak}. 
Our theorem can not be applied for $k\leq 4$. 
In the following tables, 
we formally substitute $k\leq 4$ in the formula of Theorem \ref{thm:main}. 
Hashimoto conjectured that the dimension of $S_{4,0}(\Gamma (D,1))$  
(resp. $S_{3,0}(\Gamma (D,1))$) can be obtained 
by substituing $k=4$ in Theorem \ref{thm:main} 
(resp. by substituing $k=3$ and adding $+1$). 
(Conjecture 4.3, 4.4 in \cite{Has84}). 
We see from our numerical examples that 
there is a possibility that these conjectures hold 
also in the general case $\Gamma (D_1,D_2)$. 
(cf. \cite{Ibu07} in the split case). 

\vspace{5mm} 
$\mathbf{(I) D=2\cdot 3}$  

\vspace{2mm} 
{\samepage 
(i) $D_1=2\cdot 3$, $D_2=1$ \\ 
\resizebox{\textwidth}{15mm}{ 
\begin{tabular}{|c|ccccc|ccccccccccc|} \hline 
$j\backslash k$ &0&1&2&3&4&5&6&7&8&9&10&11&12&13&14&15 \\ \hline 
0&0&-1&0&-1&2&0&4&2&8&5&15&10&25&15&34&26 \\ \hline 
2&-1&2&0&1&2&2&5&7&15&17&33&34&53&58&91&96 \\ \hline 
4&0&-1&0&2&4&6&14&19&35&42&67&77&114&126&179&200 \\ \hline 
6&-2&-1&1&5&9&17&30&40&65&82&118&145&195&224&299&341 \\ \hline 
8&-3&-2&2&7&19&27&49&67&106&131&188&223&298&346&448&514 \\ \hline 
\end{tabular}}
} 

\vspace{2mm} 
{\samepage 
(ii) $D_1=3$, $D_2=2$ \\ 
\resizebox{\textwidth}{15mm}{ 
\begin{tabular}{|c|ccccc|ccccccccccc|} \hline 
$j\backslash k$ &0&1&2&3&4&5&6&7&8&9&10&11&12&13&14&15 \\ \hline 
0&-1&-1&0&0&2&1&3&4&7&5&9&11&17&14&21&24 \\ \hline 
2&0&1&0&1&0&1&3&6&7&10&18&23&29&36&52&61 \\ \hline 
4&0&-1&0&1&2&2&7&12&19&23&36&48&65&75&100&122 \\ \hline 
6&0&0&1&5&6&11&19&29&39&51&72&93&116&140&180&214 \\ \hline 
8&-1&-2&2&5&12&16&30&44&64&79&110&139&179&211&265&315 \\ \hline 
\end{tabular}}
} 

\vspace{2mm} 
{\samepage 
(iii) $D_1=2$, $D_2=3$ \\ 
\resizebox{\textwidth}{15mm}{ 
\begin{tabular}{|c|ccccc|ccccccccccc|} \hline 
$j\backslash k$ &0&1&2&3&4&5&6&7&8&9&10&11&12&13&14&15 \\ \hline 
0&-1&-1&0&0&1&1&3&2&4&6&6&7&12&11&14&19 \\ \hline 
2&0&1&0&0&0&1&1&3&4&7&10&14&18&25&31&39 \\ \hline 
4&1&0&0&2&1&3&7&8&13&20&24&34&45&53&69&86 \\ \hline 
6&0&-1&1&3&2&8&12&16&25&36&43&60&77&92&115&143 \\ \hline 
8&0&0&2&3&9&13&21&30&43&56&75&94&119&146&178&212 \\ \hline 
\end{tabular}}
} 

\vspace{2mm} 
{\samepage  
(iv) $D_1=1$, $D_2=2\cdot 3$ \\  
\resizebox{\textwidth}{15mm}{ 
\begin{tabular}{|c|ccccc|ccccccccccc|} \hline 
$j\backslash k$ &0&1&2&3&4&5&6&7&8&9&10&11&12&13&14&15 \\ \hline 
0&0&-1&0&-1&1&2&2&2&3&4&6&6&8&8&11&13 \\ \hline 
2&-1&2&0&0&0&0&1&2&2&4&5&9&10&15&18&22 \\ \hline 
4&1&0&0&1&1&1&4&5&7&11&15&19&26&32&40&50 \\ \hline 
6&0&0&1&3&1&6&7&11&17&21&27&38&46&58&70&86 \\ \hline 
8&0&0&2&1&8&8&12&19&27&34&47&56&72&89&109&127 \\ \hline 
\end{tabular}}
}  

\vspace{3mm} 
$\mathbf{(II) D=2\cdot 5}$ 

\vspace{2mm} 
{\samepage 
(i) $D_1=2\cdot 5$, $D_2=1$ \\ 
\resizebox{\textwidth}{15mm}{ 
\begin{tabular}{|c|ccccc|ccccccccccc|} \hline
$j\backslash k$ &0&1&2&3&4&5&6&7&8&9&10&11&12&13&14&15 \\ \hline
0&0&-1&0&-1&4&2&13&5&26&19&56&41&98&70&149&123 \\ \hline
2&-2&3&0&3&9&12&28&39&82&99&170&185&285&316&470&513 \\ \hline
4&0&-3&0&8&23&33&76&99&180&227&346&408&587&675&926&1051 \\ \hline
6&-8&-7&3&18&46&83&150&203&330&423&607&742&1004&1173&1534&1771 \\ \hline
8&-22&-12&3&31&88&141&246&347&532&684&955&1157&1522&1805&2302&2669 \\ \hline
\end{tabular} 
}}

\vspace{2mm} 
{\samepage 
(ii) $D_1=5$, $D_2=2$ \\ 
\resizebox{\textwidth}{15mm}{ 
\begin{tabular}{|c|ccccc|ccccccccccc|} \hline
$j\backslash k$ &0&1&2&3&4&5&6&7&8&9&10&11&12&13&14&15 \\ \hline
0&0&-1&0&-1&2&3&7&7&15&16&30&32&53&55&84&88 \\ \hline
2&-2&3&0&1&4&8&16&28&45&61&93&118&164&203&269&316 \\ \hline
4&2&-1&0&5&13&21&45&64&102&140&201&253&344&418&539&643 \\ \hline
6&-3&-4&3&11&25&53&88&128&196&259&355&456&592&721&909&1079 \\ \hline
8&-12&-5&3&17&53&88&146&218&315&415&564&706&905&1105&1367&1616 \\ \hline
\end{tabular}}
}

\vspace{2mm} 
{\samepage 
(iii) $D_1=2$, $D_2=5$ \\ 
\resizebox{\textwidth}{15mm}{ 
\begin{tabular}{|c|ccccc|ccccccccccc|} \hline
$j\backslash k$ &0&1&2&3&4&5&6&7&8&9&10&11&12&13&14&15 \\ \hline
0&-1&-1&0&0&2&2&4&5&8&10&14&17&23&28&35&42 \\ \hline
2&-1&2&0&0&2&4&5&12&16&24&35&47&60&81&100&124 \\ \hline
4&2&0&0&2&4&7&16&24&36&53&73&96&127&160&200&247 \\ \hline
6&-1&-1&3&7&10&25&35&53&78&106&137&184&229&285&352&426 \\ \hline
8&-3&-1&3&6&23&35&57&86&122&161&218&275&347&430&524&626 \\ \hline
\end{tabular}} 
} 

\vspace{2mm} 
{\samepage 
(iv) $D_1=1$, $D_2=2\cdot 5$ \\ 
\resizebox{\textwidth}{15mm}{ 
\begin{tabular}{|c|ccccc|ccccccccccc|} \hline
$j\backslash k$ &0&1&2&3&4&5&6&7&8&9&10&11&12&13&14&15 \\ \hline
0&-1&-1&0&0&2&3&4&5&7&9&12&14&18&21&26&31 \\ \hline
2&-1&2&0&0&1&2&3&7&9&14&20&28&35&48&59&73 \\ \hline
4&2&0&0&1&2&3&9&13&20&30&42&55&74&93&117&145 \\ \hline
6&0&0&3&6&7&17&23&34&50&66&85&114&141&175&215&260 \\ \hline
8&-1&0&3&4&16&22&35&53&75&98&133&166&210&260&317&377 \\ \hline
\end{tabular}}
}

\vspace{3mm} 
$\mathbf{(III) D=3\cdot 5}$ 

\vspace{2mm} 
{\samepage 
(i) $D_1=3\cdot 5$, $D_2=1$ \\ 
\resizebox{\textwidth}{15mm}{ 
\begin{tabular}{|c|ccccc|ccccccccccc|} \hline
$j\backslash k$ &0&1&2&3&4&5&6&7&8&9&10&11&12&13&14&15 \\ \hline
0&0&-1&1&0&9&8&34&29&86&85&183&178&331&318&536&531 \\ \hline
2&-1&3&0&7&30&52&117&170&311&405&640&775&1120&1324&1821&2100 \\ \hline
4&-3&-6&1&28&84&149&298&431&703&934&1357&1694&2316&2789&3644&4283 \\ \hline
6&-29&-24&3&63&174&323&574&834&1281&1702&2373&2985&3936&4757&6044&7136 \\ \hline
8&-79&-54&6&119&330&575&979&1416&2091&2756&3752&4681&6044&7305&9117&10746 \\ \hline
\end{tabular}} 
}

\vspace{2mm} 
{\samepage 
(ii) $D_1=5$, $D_2=3$ \\ 
\resizebox{\textwidth}{15mm}{ 
\begin{tabular}{|c|ccccc|ccccccccccc|} \hline
$j\backslash k$ &0&1&2&3&4&5&6&7&8&9&10&11&12&13&14&15 \\ \hline
0&-1&-1&1&1&3&6&15&17&30&50&63&86&126&150&194&254 \\ \hline
2&0&2&0&4&9&24&44&75&115&172&239&327&429&555&699&869 \\ \hline
4&3&-3&1&14&29&63&118&176&271&388&520&698&908&1134&1426&1751 \\ \hline
6&-8&-10&3&32&64&137&229&344&503&705&927&1219&1559&1935&2384&2909 \\ \hline
8&-24&-23&6&50&131&237&390&579&827&1121&1481&1899&2397&2960&3613&4343 \\ \hline
\end{tabular}}
}

\vspace{2mm} 
{\samepage 
(iii) $D_1=3$, $D_2=5$ \\ 
\resizebox{\textwidth}{15mm}{ 
\begin{tabular}{|c|ccccc|ccccccccccc|} \hline
$j\backslash k$ &0&1&2&3&4&5&6&7&8&9&10&11&12&13&14&15 \\ \hline
0&-1&-1&1&1&5&6&11&15&24&32&45&58&78&98&124&152 \\ \hline
2&0&2&0&2&5&14&24&43&65&98&137&187&245&319&401&499 \\ \hline
4&1&-1&1&6&17&35&64&102&153&218&300&398&516&654&816&1001 \\ \hline
6&-4&-2&3&20&42&83&133&206&295&409&543&711&901&1127&1384&1681 \\ \hline
8&-12&-11&6&30&79&139&228&337&481&649&859&1099&1387&1712&2089&2509 \\ \hline
\end{tabular}}
} 

\vspace{2mm} 
{\samepage 
(iv) $D_1=1$, $D_2=3\cdot 5$ \\ 
\resizebox{\textwidth}{15mm}{ 
\begin{tabular}{|c|ccccc|ccccccccccc|} \hline
$j\backslash k$ &0&1&2&3&4&5&6&7&8&9&10&11&12&13&14&15 \\ \hline
0&0&-1&1&0&3&4&6&7&12&15&21&26&35&42&54&65 \\ \hline
2&-1&3&0&1&2&6&9&18&25&39&54&75&96&128&159&198 \\ \hline
4&3&0&1&4&8&13&28&41&61&88&121&158&208&261&326&401 \\ \hline
6&-1&0&3&11&16&37&54&84&121&166&217&289&362&453&556&676 \\ \hline
8&-3&-2&6&11&38&57&93&138&197&260&350&441&558&689&841&1004 \\ \hline
\end{tabular}} 
}

%%%%%%%%%%%%%%%%%%%%%%%%%%%%%%%%%%%%%%%%%%%%%%%%%%%%%%%%%%%%%%%%%
%%%%%%%%%%%%%%%%%%%%%%%%%%%%%%%%%%%%%%%%%%%%%%%%%%%%%%%%%%%%%%%%%

\vspace*{5mm}
\noindent
Hidetaka Kitayama\\
Department of Mathematics\\
Osaka University\\ 
Machikaneyama 1-1, Toyonaka\\ 
Osaka, 560-0043, Japan\\
E-mail: \texttt{h-kitayama@cr.math.sci.osaka-u.ac.jp}\\

\end{document}